\documentclass[preprint,11pt]{amsart}
\usepackage[T1]{fontenc}
\usepackage[latin9]{inputenc}
\usepackage{multirow}
\usepackage{array} 
\usepackage{amsmath, amssymb}
\usepackage{empheq}
\usepackage{amsthm} 

 \newtheorem{theorem}{Theorem}[section]

\newtheorem{proposition}[theorem]{Proposition}
\newtheorem{remark}[theorem]{Remark}
\newtheorem{definition}[theorem]{Definition}
\newtheorem{corollary}[theorem]{Corollary}
 
\usepackage{color}
\usepackage{graphicx,color}
\usepackage{graphicx} 
\usepackage{subcaption}
\usepackage{amsmath, amssymb, graphics}

\def \r{\mathbb R}

\def \n{\mathbb N}
\def \z{\mathbb Z}
\def \L{\mathbb L}

\title{Stationary soap films with vertical potentials}
\author{Rafael L\'opez}
\email{rcamino@ugr.es}
\address{ Departamento de Geometr\'{\i}a y Topolog\'{\i}a\\
Universidad de Granada\\
18071 Granada, Spain}
\author{\'Alvaro P\'ampano}
\email{alvaro.pampano@ttu.edu}
\address{Department of Mathematics and Statistics\\ 
Texas Tech University\\ 
Lubbock, TX, 79409, USA}

\begin{document}

\begin{abstract}
We classify cylindrical surfaces in the Euclidean space whose mean curvature is a $n$th-power of the distance to a reference plane. The generating curves of these surfaces, called $n$-elastic curves, have a variational characterization as critical points of a curvature energy generalizing the classical elastic energy. We give a full description of such curves obtaining, in some particular cases, closed curves including simple ones.
\end{abstract}
 \maketitle

\section{Introduction}\label{1}

We consider the equilibrium shape of a (possibly incompressible) fluid volume $\Omega$ of constant mass density contained in the Euclidean 3-space $\mathbb{R}^3$ with a potential energy depending on the height $z$. The boundary of the fluid bulk $\Omega$ will be regarded as an immersed smooth surface $\Sigma$ modeling the free interface between the interior and ambient fluids. 

In this setting, the free surface energy is proportional to the surface area $\mathcal{A}[\Sigma]$, while the incompressibility condition of the fluid volume can be included as a Lagrange multiplier fixing the enclosed volume $\mathcal{V}[\Omega]$. When the domain $\Omega$ is not closed, $\mathcal{V}[\Omega]$ will represent the algebraic volume between the surface $\Sigma$ and the plane $z=0$. Similarly, if $\Omega$ is not embedded, $\mathcal{V}[\Omega]$ will be regarded as the signed algebraic volume. Finally, to account for the potential energy an extra term will be added to the total energy which has the expression
\begin{equation}\label{energy}
E[\Sigma]=\sigma\mathcal{A}[\Sigma]+\eta\int_\Omega f(z)\,dV+\varpi\mathcal{V}[\Omega]\,,
\end{equation}
where $f$ is a smooth function depending on the height $z$ and is defined on a suitable domain $\Omega\subset\mathbb{R}^3$ occupied by the fluid, whose boundary is described by the surface $\Sigma$. The energy parameters $\sigma>0$, $\eta\in\mathbb{R}$ and $\varpi\in\mathbb{R}$ are constants motivated by the physical applications. To be precise, the parameter $\sigma>0$ represents the surface tension, $\eta$ is a constant depending on the difference between the mass densities of the interior and ambient fluids and $\varpi$ acts as a Lagrange multiplier which enforces incompressibility, in such a way that if $\varpi=0$ there is no volume constraint. 

One of the physically most relevant cases appears when $f(z)=z$ because the interface models a homogeneous liquid drop adhering to a horizontal plane under the action of constant gravity. When $\eta<0$ it corresponds to a sessile drop, while if $\eta>0$ we obtain a pendent drop. In absence of gravity ($\eta=0$), the surface has constant mean curvature. We refer to the book of Finn for details (\cite{fi}). Also in early works of Serrin and Wente we can find different motivations for considering potential energies depending on one space coordinate (\cite{se,we}). In this paper we consider other possible potentials for arbitrary functions depending on the height $z$, which may give rise to different physical scenarios, as for example $f(z)=z^{-1}$ which represents a potential energy associated to an inverse-square law force. 

The equilibria for the energy $E[\Sigma]$ can be obtained by the balance between the capillary force that comes from the surface tension and the force associated with the potential energy acting on the fluid volume. Therefore, the equilibrium shapes are governed by the Young-Laplace equation
\begin{equation}\label{energy2}
2\sigma H=\eta f(z)+\varpi\quad\quad\quad\text{on $\Sigma$}\,,
\end{equation}
where $H$ denotes the mean curvature of the surface $\Sigma$. In Section \ref{2}, we obtain this condition as the Euler-Lagrange equation for the associated variational problem (Proposition \ref{EulerLagrange}) and use it to show that in many cases there are not closed and embedded equilibria (Proposition \ref{pr-closed}).

In the rest of the paper we focus on equilibria that are invariant in one space direction. In such a case the Young-Laplace equation \eqref{energy2} reduces to a problem of planar curves whose curvature depends on the distance to a fixed straight line.  The study of these types of problems goes back to the 17th century when Bernouilli analyzed the \emph{lintearia} which is the shape of a long cloth sheet full of water, obtaining a relation with the classical elastic curves (\cite{Truesdell}). This particular case corresponds with the choice $f(z)=z$ in \eqref{energy2}.

Planar curves with prescribed curvature also have interest in physics by themselves, as for instance, to understand some processes in dynamics of plasmas the motion of charged particles in specified fields are studied (\cite{bi,li}). We briefly describe this application in what follows. In classical physics, the equation of motion $t\mapsto q(t)=(x(t),y(t),z(t))\in\mathbb{R}^{3}$ for a (nonrelativistic) particle of mass $m$ and charge $e$ under the action of a magnetic field $\mathbf{B}$ is given by the Newton-Lorentz law 
\begin{equation}
\label{eq-phy1}
m\ddot q=e\left(\dot q\times \mathbf{B}\right),
\end{equation}
where the upper dot denotes the derivative with respect to time $t$. In general, these equations cannot be integrated analytically and the associated trajectories are very complicated, with the exception of some particular cases, such as when the vector field $\mathbf{B}$ is uniform.  Assume that $\mathbf{B}$ is parallel to a fixed direction and that its magnitude depends on the distance to a plane parallel to this direction. After a change of coordinates we may suppose that the direction is $e_1=(1,0,0)$ and that the plane is the $xy$-plane. Then $\mathbf{B}(x,y,z)=B(z)e_1$. Writing down explicitly each coordinate of the vector equation \eqref{eq-phy1} for the magnetic field $\mathbf{B}(x,y,z)=(B(z),0,0)\in\mathbb{R}^{3}$, we obtain the system of second order differential equations
\begin{equation}
\label{eq-phy2}
\left\{\begin{array}{l}
\ddot x(t)=0\\
\ddot y(t)=\dfrac{e}{m}B(z)\dot z(t)\\
\ddot z(t)=-\dfrac{e}{m}B(z)\dot y(t)\,.
\end{array}\right.
\end{equation}
From the first equation, $x(t)$ describes a uniform motion and the projection of $q(t)$ on the $yz$-plane is a trajectory which only depends on the $z$-coordinate. If we denote by $q(t)$ again this planar curve and after the change of variables  $\gamma(t)=q(-mt/e)$, the system \eqref{eq-phy2} reduces to $\ddot \gamma=J B(z)\dot \gamma$, where $J$ 
is the counter-clockwise rotation of angle $\pi/2$ in the $yz$-plane. It turns out that this equation can be viewed as a problem of prescribing the curvature for planar curves. Indeed, observe first that the velocity $\lVert\dot\gamma\rVert$ of $\gamma$ is constant since 
\[\frac{d}{d t}\lVert \dot \gamma\rVert^{2}=2\langle\ddot\gamma(t),\dot\gamma(t)\rangle=2 B(z)\langle J \dot\gamma(t),\dot\gamma(t)\rangle=0\,.\]
Second, since the curvature of $\gamma$ is 
\[\kappa(t)=\frac{\langle \ddot \gamma(t),J\dot\gamma(t)\rangle}{\lVert\dot\gamma(t)\rVert^3}\,,\]
we obtain that
\[\kappa(t)=\frac{B(z)}{\lVert\dot\gamma(t)\rVert}\,.\]
For instance, if $\gamma$ has unit velocity, then $\kappa(s)=B(z)$. As a first model for the motion of plasma, bounded or even closed trajectories deserve further investigation. A special case occurs when $B$ is the identity so that $\kappa(s)=z(s)$ obtaining elastic curves from the classical theory of Bernouilli and Euler (\cite{Euler}).

More generally,  we will study planar curves whose curvature satisfies 
$$\kappa(s)=z^n+\mu\,,$$ 
with $n, \mu\in\r$. We will call these curves \emph{$n$-elastic curves} and they will also arise as the generating curves of right cylinders satisfying \eqref{energy2} for $f(z)=z^n$. In Section \ref{3}, we will prove that $n$-elastic curves are solutions of a variational problem involving energy functionals depending on the curvature (Theorem \ref{crit}). Section \ref{4} is devoted to the analysis of the geometric properties related to symmetries of $n$-elastic curves (Propositions \ref{pr6} and \ref{pr1}) while on Section \ref{seccer} we investigate the existence of closed curves (Proposition \ref{pr3} and Theorem \ref{t-closed}). Finally, in Section \ref{5} we classify the shapes of $n$-elastic curves giving a complete catalog of all the possible types. Besides some horizontal straight lines (Proposition \ref{lines}), among curves whose arc length parameter is defined on the entire real line, which will be called {\it complete} curves, we obtain families of $n$-elastic curves which imitate all the shapes of Euler's classification of elastic curves (Figure \ref{figtipos}) as well as different families of curves (Figure \ref{figtipos2}).

Following the terminology of classical elastic curves (see, for instance, the lecture notes of Singer \cite{si}) when the curvature is periodic we will distinguish two families of curves, {\it orbitlike $n$-elastic curves} defined by the property that their periodic curvature has constant sign (see Figure \ref{figtipos}, (a)), and {\it wavelike $n$-elastic curves} which are those curves whose curvature $\kappa$ oscillates between a value $\kappa_0$ and $-\kappa_0$ increasing and decreasing as the parameter goes in the domain. Among the family of wavelike $n$-elastic curves we may find \emph{multiloops} (Figure \ref{figtipos}, (c)), \emph{pseudo-lemniscates} (Figure \ref{figtipos}, (d)), \emph{deep waves} (both self-intersecting and simple, Figure \ref{figtipos}, (e) and (f), respectively), \emph{rectangular} $n$-elastic curves (Figure \ref{figtipos}, (g)) and \emph{shallow waves} (Figure \ref{figtipos}, (h)). 

As in the classical theory of elastic curves, in between the wavelike and orbitlike families of $n$-elastic curves, we find the {\it borderline $n$-elastic curve}, which has nonperiodic curvature and asymptotically approaches a horizontal $n$-elastic line (Figure \ref{figtipos}, (b)).  

\begin{figure}[hbtp]
\centering
\begin{subfigure}[b]{0.22\linewidth}
\includegraphics[width=\linewidth]{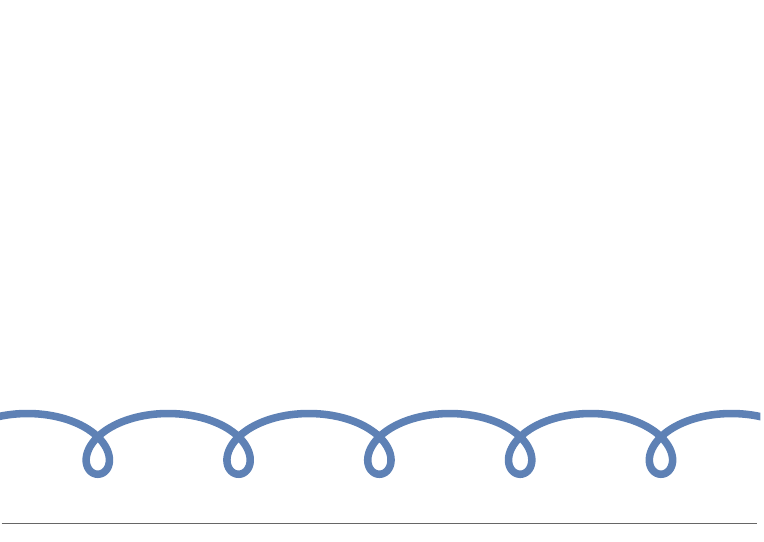}
\caption{Orbitlike}
\end{subfigure}
\quad
\begin{subfigure}[b]{0.22\linewidth}
\includegraphics[width=\linewidth]{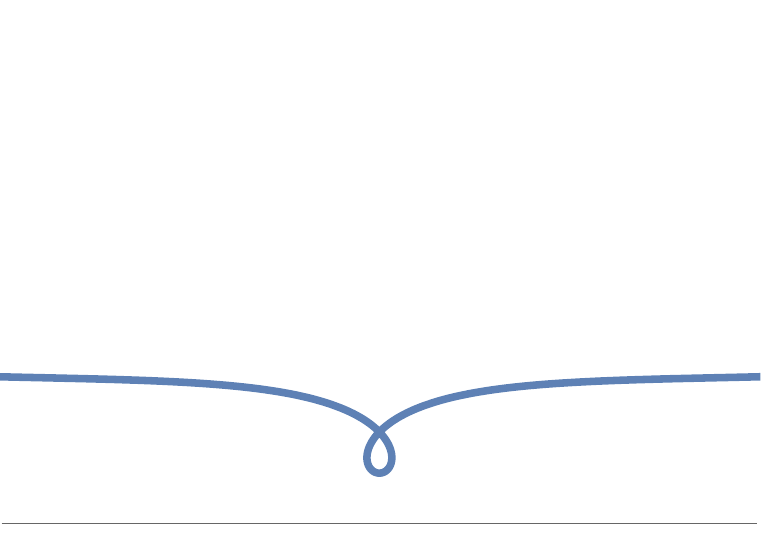}
\caption{Borderline}
\end{subfigure}
\quad
\begin{subfigure}[b]{0.22\linewidth}
\includegraphics[width=\linewidth]{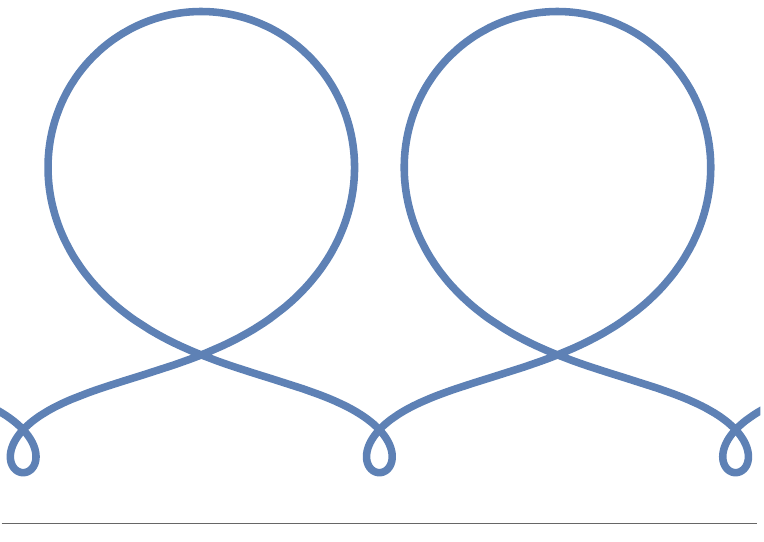}
\caption{Multiloop}
\end{subfigure}
\quad
\begin{subfigure}[b]{0.22\linewidth}
\includegraphics[width=\textwidth]{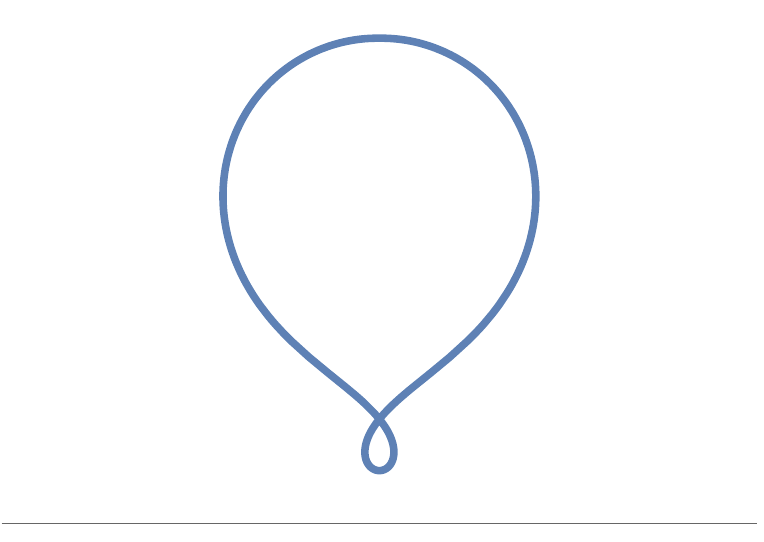}
\caption{Ps.-Lemniscate}
\end{subfigure}
\\
\begin{subfigure}[b]{0.22\linewidth}
\includegraphics[width=\textwidth]{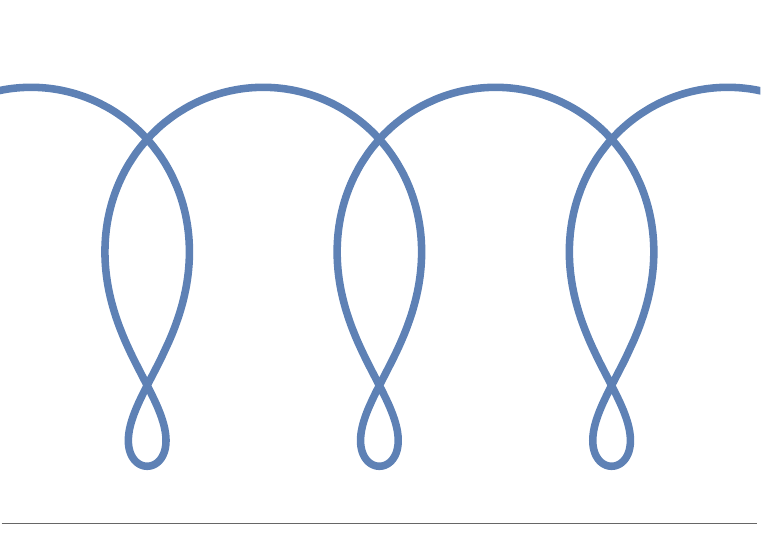}
\caption{Deep Waves}
\end{subfigure}
\quad
\begin{subfigure}[b]{0.22\linewidth}
\includegraphics[width=\textwidth]{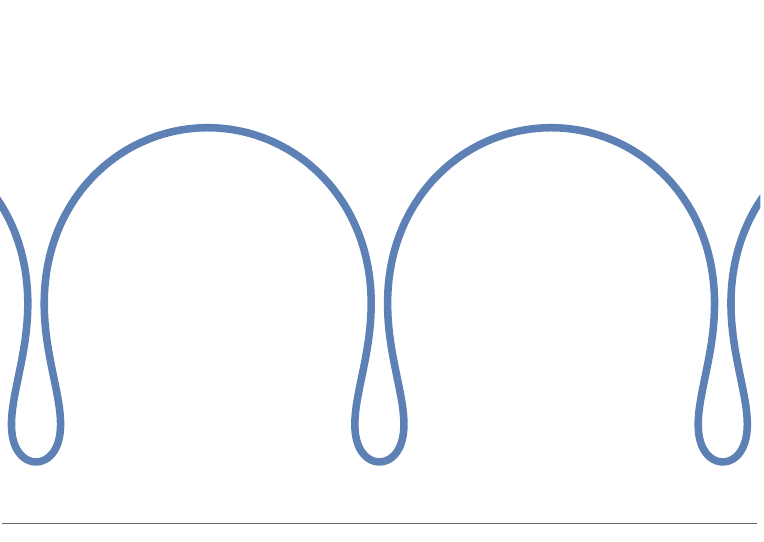}
\caption{Deep Waves}
\end{subfigure}
\quad
\begin{subfigure}[b]{0.22\linewidth}
\includegraphics[width=\textwidth]{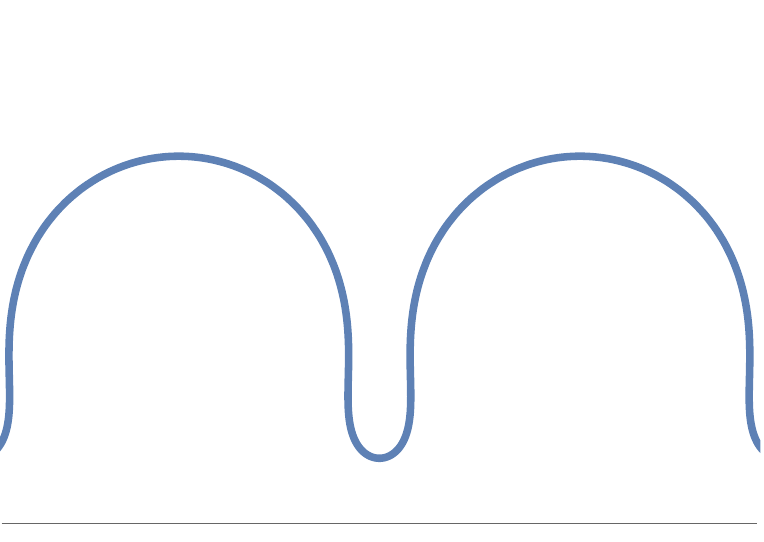}
\caption{Rectangular}
\end{subfigure}
\quad 
\begin{subfigure}[b]{0.22\linewidth}
\includegraphics[width=\textwidth]{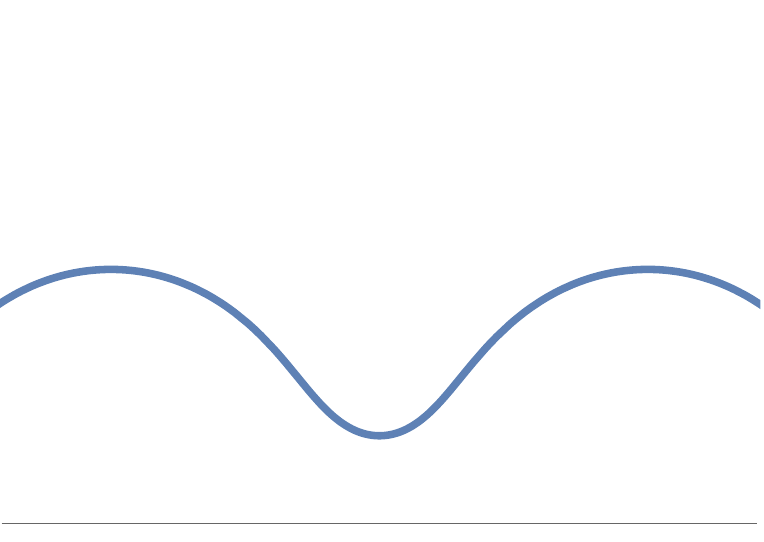}
\caption{Shallow Waves}
\end{subfigure}
\caption{A family of $n$-elastic curves for $n=-2.5$ and $\mu=-1$.}\label{figtipos}
\end{figure}
 
All these shapes imitate the types of Euler's elasticae (see for example \cite{Euler} and, more recently, \cite{mi}) although the equation of the classical elastica, $2\kappa''+\kappa^3-\lambda\kappa=0$ for some constant $\lambda$, is completely different for arbitrary choices of $n$ and $\mu$. Observe that $n$-elastic curves of above types may have their loops pointing towards the other direction. This is the case, for instance, of the classical elastic curves ($n=1$ and $\mu=0$). Moreover, if $n\in\n$, $n$-elastic curves may cut the $y$-line, i.e., the line of equation $z=0$, as well (compare, once again, with classical elastic curves). In particular, if $n\in\n$ is even, among the $n$-elastic curves which cut the $y$-line, we prove the existence of closed curves (Theorem \ref{t-closed}). Some of these curves are also simple, which gives a first difference with respect to the theory of classical elastic curves. See Figure \ref{cerradas}.

\begin{figure}[hbtp]
\centering
\makebox[\textwidth][c]{
\includegraphics[width=.2\textwidth]{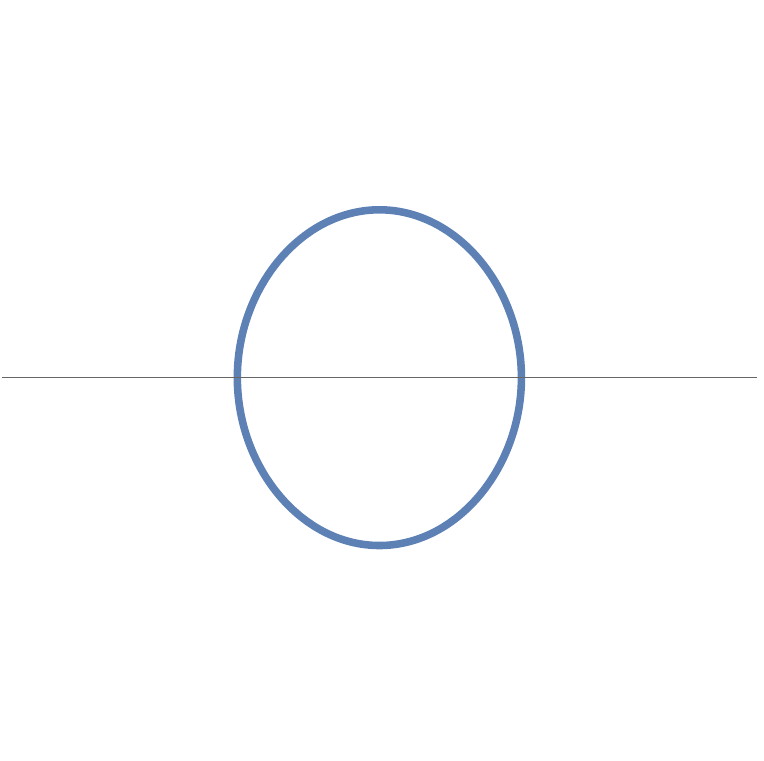}\,\includegraphics[width=.2\textwidth]{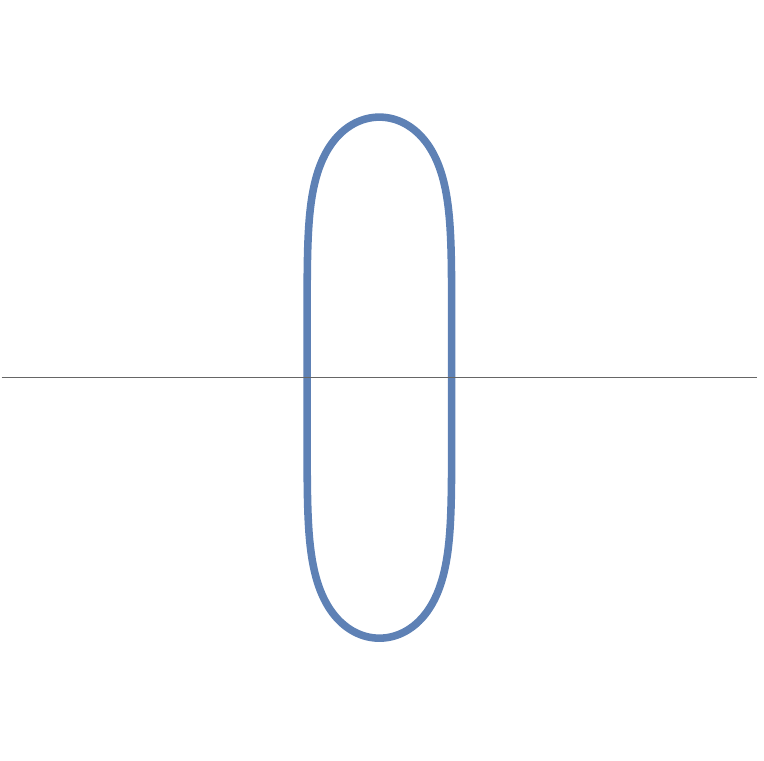}\, \includegraphics[width=.2\textwidth]{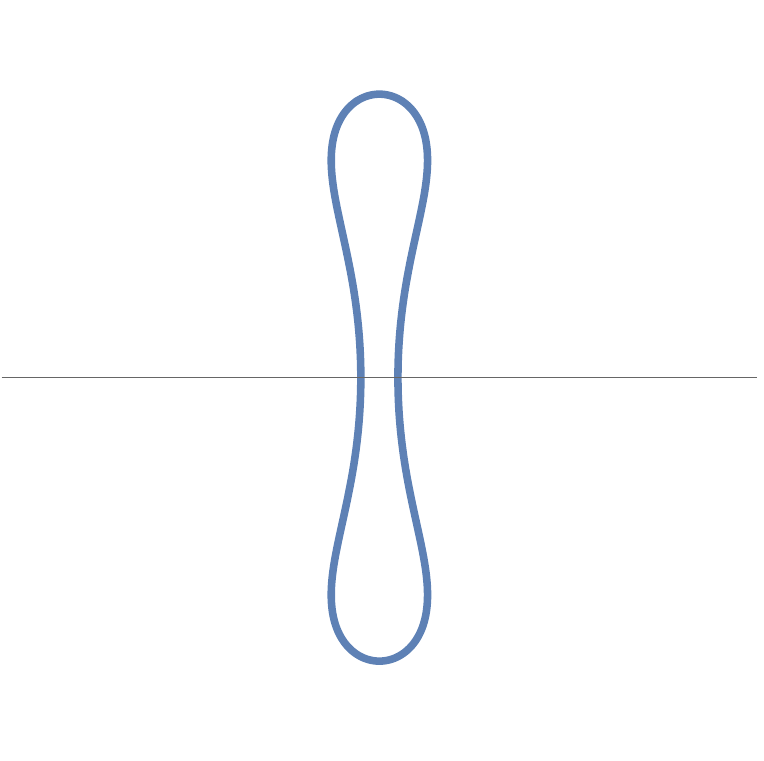}\, \includegraphics[width=.2\textwidth]{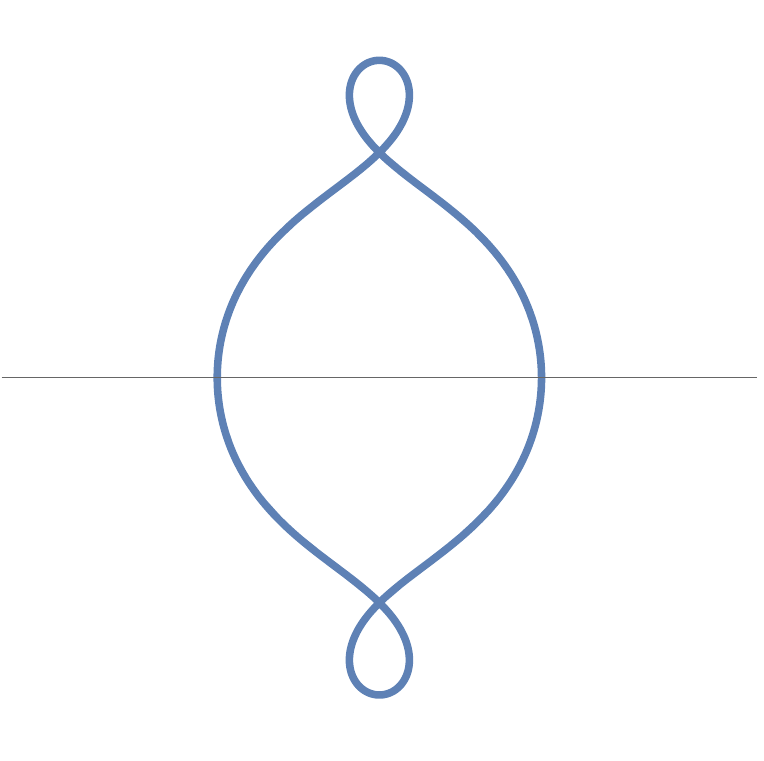}\, \includegraphics[width=.2\textwidth]{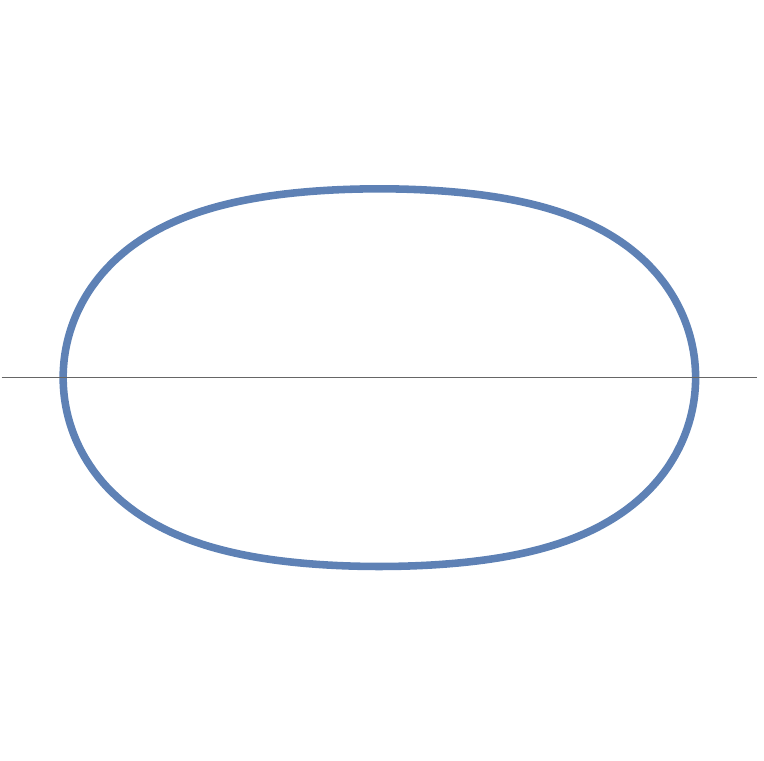}}
\caption{Closed $n$-elastic curves for $n=4$. From left to right: $\mu=1$, $\mu=0$, $\mu=-0.35$, $\mu=-1$ and $\mu=-1.2$.}\label{cerradas}
\end{figure}

Apart from this first difference regarding closed curves, among the general case of $n$-elastic curves, there is a family consisting on curves essentially different to above cases. Indeed, we have nonperiodic $n$-elastic curves that are not borderline. We call them \emph{catenary-like $n$-elastic curves}, since when $n=-2$ and $\mu=0$ we find the catenary. Some of these curves are simple while others have self-intersections (Figure \ref{figtipos2}, (c) and (d)). The simple ones are graphs over the $y$-line, either defined on a bounded interval (Figure \ref{figtipos2}, (e)) or entire graphs (Figure \ref{figtipos2}, (f)-(h)).    

\begin{figure}[hbtp]
\centering
\begin{subfigure}[b]{0.22\linewidth}
\includegraphics[width=\textwidth]{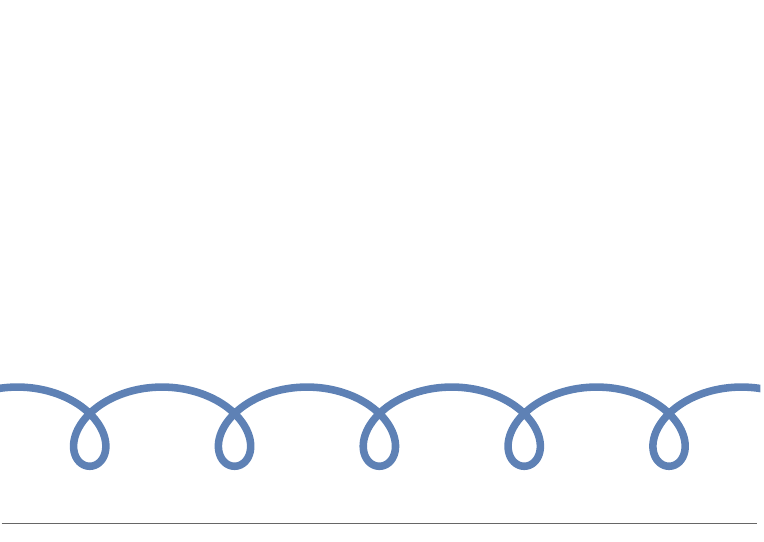}
\caption{}
\end{subfigure}
\quad
\begin{subfigure}[b]{0.22\linewidth}
\includegraphics[width=\textwidth]{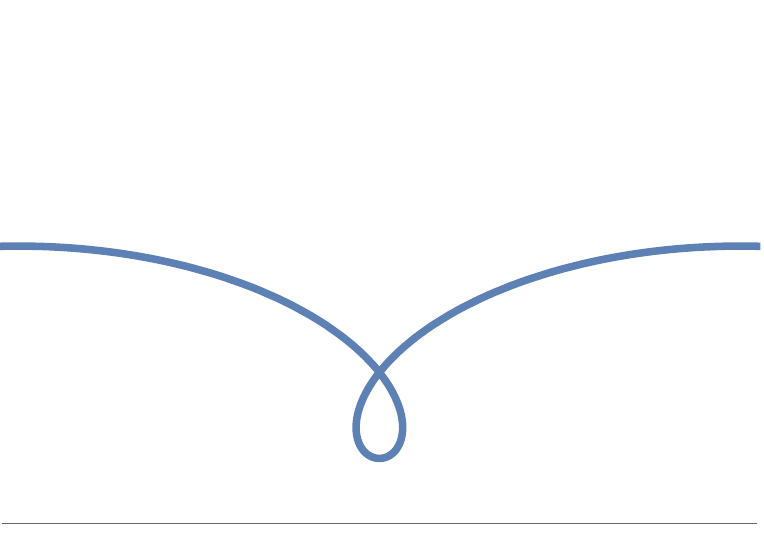}
\caption{}
\end{subfigure}
\quad
\begin{subfigure}[b]{0.22\linewidth}
\includegraphics[width=\textwidth]{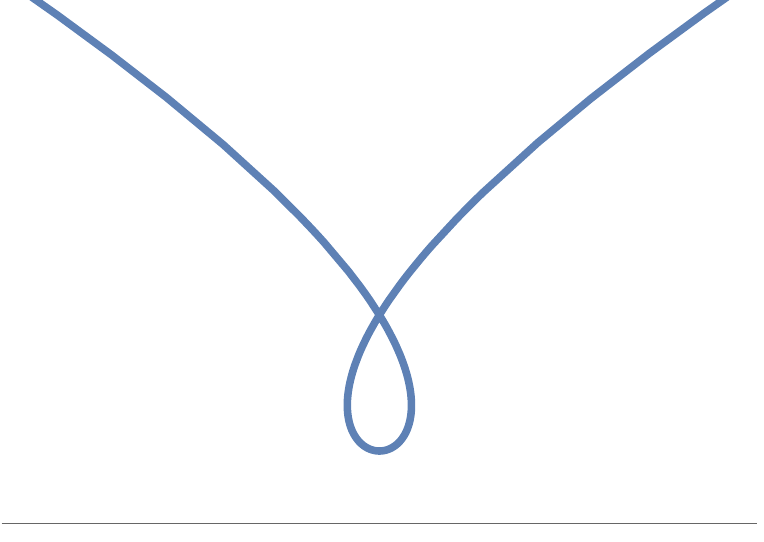}
\caption{}
\end{subfigure}
\quad
\begin{subfigure}[b]{0.22\linewidth}
\includegraphics[width=\textwidth]{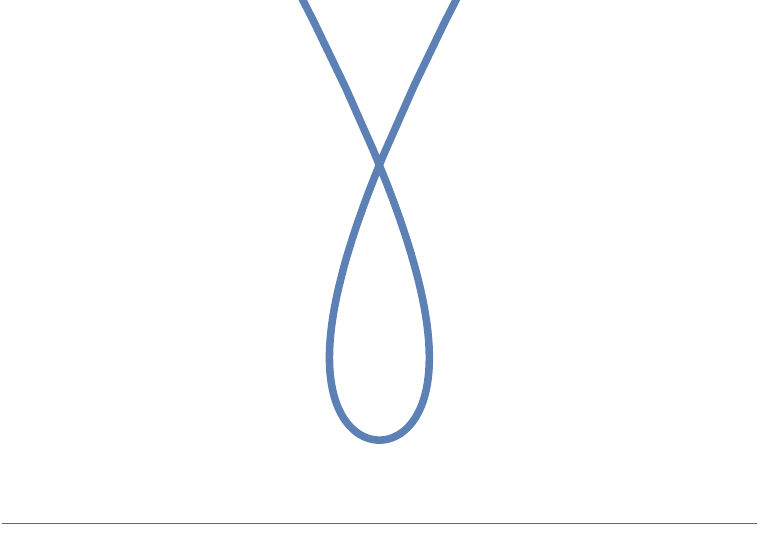}
\caption{}
\end{subfigure}
\\
\begin{subfigure}[b]{0.22\linewidth}
\includegraphics[width=\textwidth]{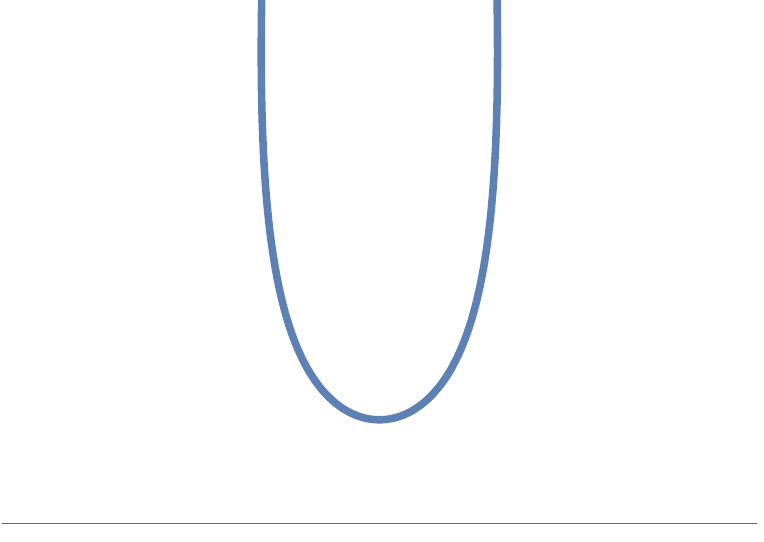}
\caption{}
\end{subfigure}
\quad
\begin{subfigure}[b]{0.22\linewidth}
\includegraphics[width=\textwidth]{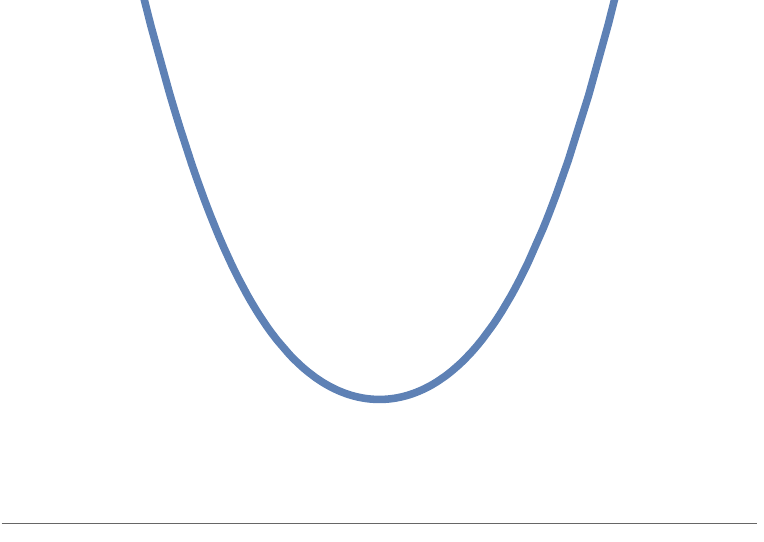}
\caption{}
\end{subfigure}
\quad
\begin{subfigure}[b]{0.22\linewidth}
\includegraphics[width=\textwidth]{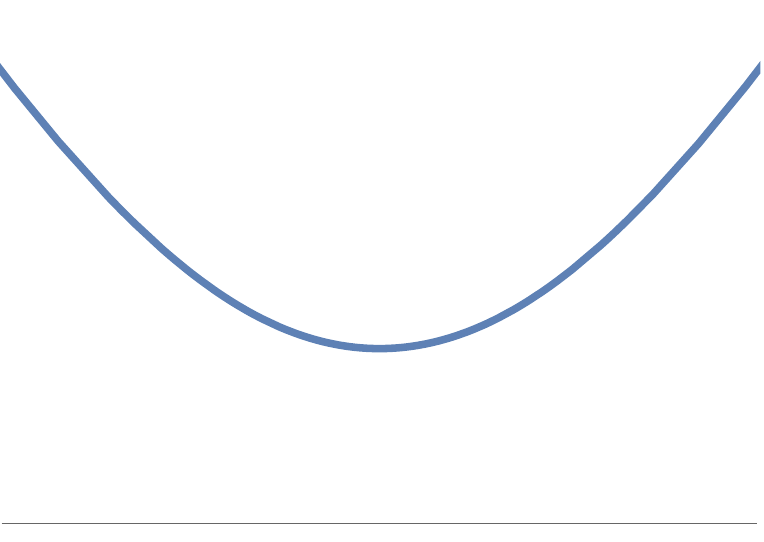}
\caption{}
\end{subfigure}
\quad
\begin{subfigure}[b]{0.22\linewidth}
\includegraphics[width=\textwidth]{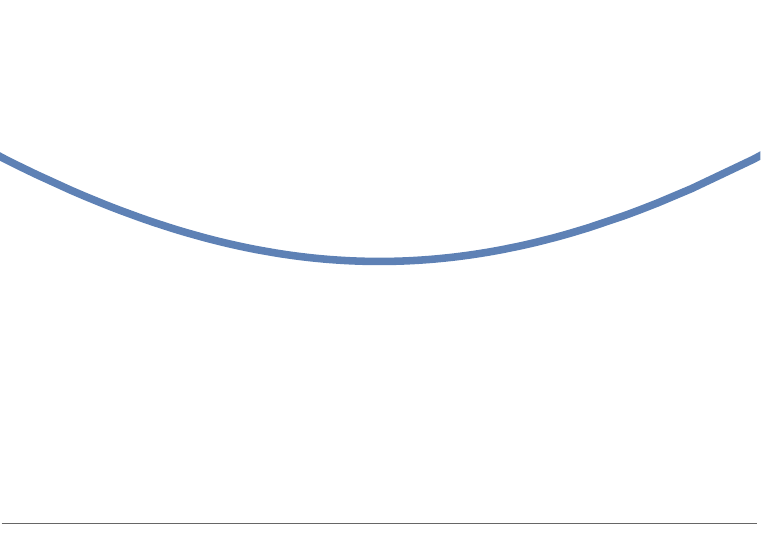}
\caption{}
\end{subfigure}
\caption{A family of $n$-elastic curves for $n=-2.5$ and $\mu=0$.}\label{figtipos2}
\end{figure}

All the curves from Figures \ref{figtipos} and \ref{figtipos2} are complete. However, if $n>-1$ and $n\notin\n$, $n$-elastic curves are not defined at $z=0$, but they may approach this line (Corollary \ref{n1}). Consequently, curves which are not complete and have shapes like in Figure \ref{cuts} can also be obtained. These curves are parts of previous complete curves.

\begin{figure}[hbtp]
\centering
\includegraphics[width=.22\textwidth]{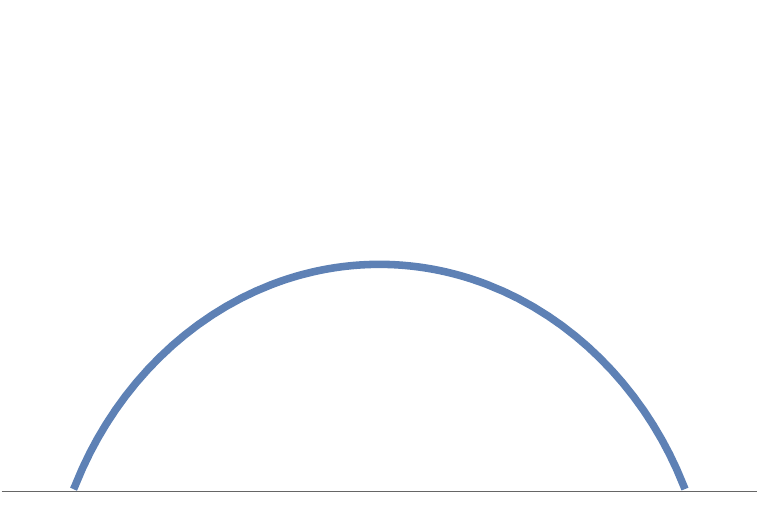}\quad\, \includegraphics[width=.22\textwidth]{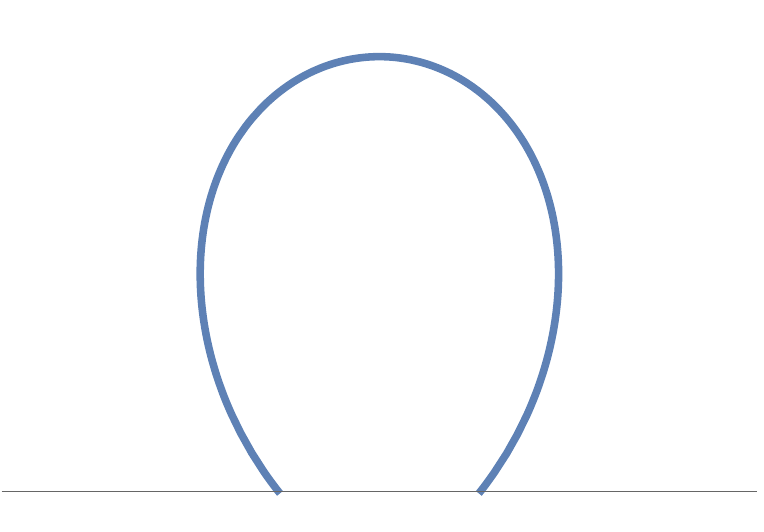}\quad\, \includegraphics[width=.22\textwidth]{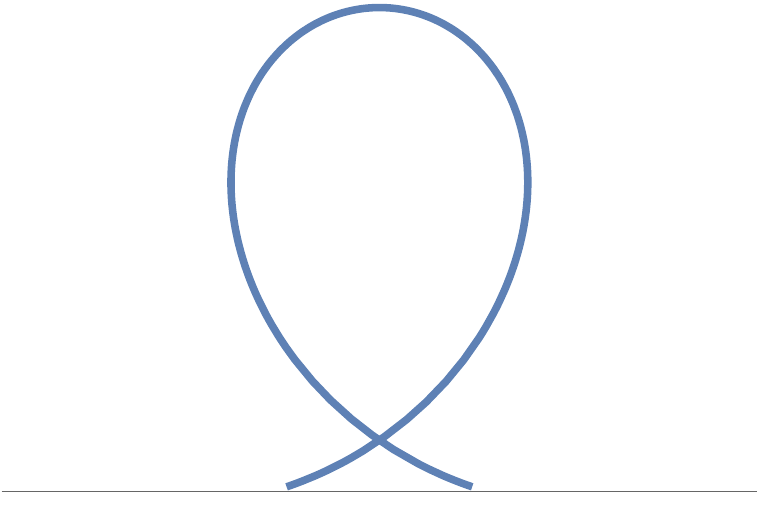}\quad\,\includegraphics[width=.22\textwidth]{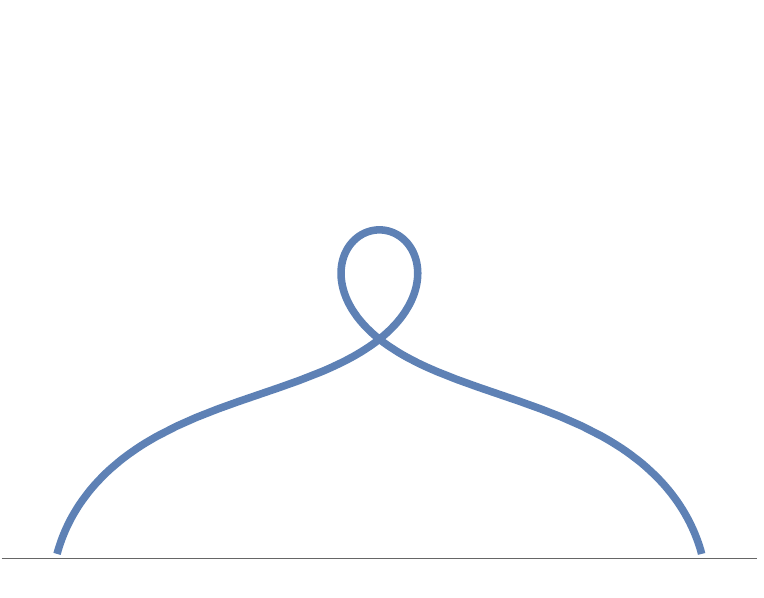}
\caption{Noncomplete $n$-elastic curves. Here $n=2.5$ and $\mu=1$ ($\mu=-1$ in the right one).}\label{cuts}
\end{figure}

All the figures of this paper have been obtained with the help of the software \textit{Mathematica} using the \texttt{NDSolve} command (no specific package). In all of them we show the $n$-elastic curve together with the $y$-line.

\section{Variational formulation of the problem}\label{2}

Let $\mathbb{R}^3$ be the Euclidean $3$-space with coordinates $(x,y,z)$ and $X:\Sigma\rightarrow\mathbb{R}^3$ be a smooth immersion of an oriented surface (with or without boundary) $\Sigma$. When the context is clear, no distinction will be made between the abstract surface $\Sigma$ and its image $X(\Sigma)\subset\mathbb{R}^3$. We denote by $\nu\colon\Sigma\rightarrow\mathbb{S}^2\subset\mathbb{R}^3$ the associated Gauss map, which will be identified with the (globally defined) unit normal vector field along $\Sigma$. 

The energy $E[\Sigma]$ in \eqref{energy} is a linear combination of the surface area $\mathcal{A}[\Sigma]$,  the potential energy depending on the height $z$ and the (signed) algebraic volume between the surface $\Sigma$ and the plane $z=0$. For an immersion $X\colon\Sigma\rightarrow\mathbb{R}^3$ the surface area is defined by 
$$\mathcal{A}[\Sigma]=\int_\Sigma\,d\Sigma\,.$$
On the other hand, if we denote by $W=z e_3$, where $e_3=(0,0,1)$, the divergence of $W$ is $\mbox{div}\,W=1$. Similarly, let $g(z)$ be a function such that $g'(z)=f(z)$ and define $\widetilde{W}=g(z)e_3$. Then $\mbox{div}\,\widetilde{W}=f(z)$. Therefore, after applying the divergence theorem, the other terms in the energy $E[\Sigma]$ are defined by
$$\int_\Omega f(z)\,dV=\int_\Sigma \langle\widetilde{W},\nu\rangle\,d\Sigma\,,\quad\quad\quad \mathcal{V}[\Omega]=\int_\Sigma \langle W,\nu\rangle\, d\Sigma\,,$$
where $\Omega$ is the domain in $\mathbb{R}^3$ occupied by the bulk of the fluid volume. Consequently, for the immersion $X\colon\Sigma\rightarrow\mathbb{R}^3$, we consider the total energy 
\begin{equation}\label{E}
E[\Sigma]=\sigma\int_\Sigma\,d\Sigma+\eta\int_\Sigma g(z)\,\langle\nu,e_3\rangle\,d\Sigma+\varpi\int_\Sigma z\,\langle\nu,e_3\rangle\,d\Sigma\,,
\end{equation}
where the constants $\sigma>0$ and $\eta$, $\varpi\in\mathbb{R}$ are fixed.

We find the Euler-Lagrange equation associated to $E[\Sigma]$ considering a one-parameter family of variations $X\colon\Sigma\times\left(-t,t\right)\rightarrow\mathbb{R}^3$,  $t>0$,  of the initial immersion $X$ defined by   $Y(-,\epsilon)=X+\epsilon\,\delta X+\mathcal{O}(\epsilon^2)$ for some sufficiently smooth variation vector field $\delta X$. We denote the variation of the functional $E[\Sigma]$ by
$$\delta E[\Sigma]=\frac{d}{d\epsilon}{{\Big|}_{\epsilon=0}}E[X(\epsilon)]\,.$$
In order to obtain the Euler-Lagrange equation characterizing equilibria on $\Sigma$, it is enough to consider compactly supported normal variations. Let $\psi\in C_0^\infty(\Sigma)$ and consider the variation vector field $\delta X=\psi\nu$. Then,  using the standard formula $\delta d\Sigma=-2H\psi d\Sigma$ (see \cite[p. 16]{Rafa}) we obtain the first variation of the first term in \eqref{E},
$$\delta\left(\int_\Sigma\,d\Sigma\right)=-2\int_\Sigma H\,\psi\,d\Sigma\,.$$
For the second term in \eqref{E} we have
$$\delta\left(\int_\Sigma g(z)\langle\nu,e_3\rangle\,d\Sigma\right)=\int_\Sigma\left(g'(z)\delta z\langle\nu,e_3\rangle+g(z)\langle\delta\nu,e_3\rangle-2g(z)\langle\nu,e_3\rangle H\,\psi\right)d\Sigma\,,$$
where we have used once again $\delta d\Sigma$. Next, we conclude from $\delta z=\langle\delta X,e_3\rangle=\psi\langle\nu,e_3\rangle$ and $\delta\nu=-\nabla\psi$ that
\begin{eqnarray*}
\delta\left(\int_\Sigma g(z)\,\langle\nu,e_3\rangle\,d\Sigma\right)&=&\int_\Sigma\left(g'(z)\langle\nu,e_3\rangle^2\psi-g(z)\langle\nabla\psi,e_3\rangle-2g(z)\langle\nu,e_3\rangle H\psi\right)d\Sigma\\
&=&\int_\Sigma\left(g'(z)\left(\langle\nu,e_3\rangle^2+\lVert\nabla z\rVert^2\right)+g(z)\Delta z-2g(z)\langle\nu,e_3\rangle H\right)\psi\,d\Sigma\\
&=&\int_\Sigma g'(z)\psi\,d\Sigma\,,
\end{eqnarray*}
where in the second line we have integrated by parts the second term and in the third one we have applied the classical formula $\Delta X=2H\nu$ (for details, see \cite[p. 29]{Rafa}). Similarly, the variation of the last term of \eqref{E} is (repeat above computations for $g(z)=z$)
$$\delta\left(\int_\Sigma z\,\langle\nu,e_3\rangle\,d\Sigma\right)=\int_\Sigma\psi\,d\Sigma\,.$$
Consequently, combining everything, the first variation formula for $E[\Sigma]$ is given by
$$\delta E[\Sigma]=\int_\Sigma\left(-2\sigma H+\eta f(z)+\varpi\right)\psi\,d\Sigma\,.$$
The Fundamental Lemma of Calculus of Variations then gives the necessary condition to be satisfied along equilibria, the so-called Euler-Lagrange equation. We sum up this in the following proposition.

\begin{proposition}\label{EulerLagrange} Let $X\colon\Sigma\rightarrow\mathbb{R}^3$ be an equilibrium immersion for the energy $E[\Sigma]$. Then, regardless of the boundary conditions, the equation 
\begin{equation}\label{EL}
2\sigma H=\eta f(z)+\varpi\,,
\end{equation}
must hold on $\Sigma$.
\end{proposition}

Previously we have shown a couple of interesting choices for $f(z)$. One case is to consider $f$ constant giving rise to surfaces with constant mean curvature. In the particular case that the densities in both sides of the interface $\Sigma$ coincide, then $\Sigma$ is a minimal surface. The case of constant mean curvature surfaces also appears if we consider $\eta=0$. Constant mean curvature surfaces have been widely studied in the literature (we mention here \cite{Rafa} and the references therein) and, hence, from now on we will discard this case. Another relevant case appears when $f(z)=z$ because $\Sigma$ models a liquid drop or, more generally a fluid bubble, supporting or hanging from a plane orthogonal to the $z$-direction under the action of constant gravity (\cite{KP}). For a general function $f(z)$, using a process of reflection and applying the Hopf's maximum principle, Wente deduced that an embedded surface inherits some symmetries of its boundary (\cite{we}); see also similar results in the nonparametric case in \cite{se}.

An interesting problem which merits investigation is whether or not there exist closed surfaces satisfying \eqref{EL}. A rescaling argument proves that in many cases these closed surfaces cannot exist.

\begin{proposition} Let $X:\Sigma\rightarrow\mathbb{R}^3$ be the immersion of a closed surface critical for the energy $E[\Sigma]$. Then,
$$2\sigma\mathcal{A}[\Sigma]+\eta\int_\Omega\left(zf'(z)+3f(z)\right)dV+3\varpi\mathcal{V}[\Omega]=0\,,$$
where $\Omega$ is the open domain of $\mathbb{R}^3$ bounded by $\Sigma$. In particular, if $f(z)=z^n$ and $\eta(n+3)z^n+3\varpi\geq 0$, there are not closed critical surfaces.
\end{proposition}
\begin{proof}
Consider a rescaling of the surface $\Sigma\mapsto r\Sigma$ for $r>0$. Then, the energy of $r\Sigma$ is given by
$$E[r\Sigma]=\sigma r^2\mathcal{A}[\Sigma]+\eta\int_\Omega f(rz)r^3 dV+\varpi r^3\mathcal{V}[\Omega]\,.$$
Thus, since $\Sigma$ ($r=1$) is a critical surface, differentiating with respect to the rescaling parameter, we get
$$0=\frac{dE}{dr}[\Sigma]=2\sigma\mathcal{A}[\Sigma]+\eta\int_\Omega\left(zf'(z)+3f(z)\right)dV+3\varpi\mathcal{V}[\Omega]\,,$$
proving the result. The second statement follows directly.
\end{proof}

If we also seek embedded surfaces, existence is even more restricted. The problem of finding closed embedded surfaces satisfying \eqref{EL} is motivated by the classical Alexandrov's result which asserts that the only embedded closed constant mean curvature surface is the round sphere (\cite{Alexandrov}). In the case that $f(z)=z$, there are not closed embedded surfaces (\cite{KP}). These results are generalized for suitable choices of $f(z)$ in the following proposition.

\begin{proposition}\label{pr-closed} If $f(z)$ is increasing (or decreasing) almost everywhere, there are not closed embedded surfaces whose mean curvature $H$ satisfies $2\sigma H=\eta f(z)+\varpi$. In particular, the result holds for $f(z)=z^n$ and $n$ not even.
\end{proposition}
\begin{proof} By contradiction, suppose that $\Sigma$ is a closed embedded surface in $\mathbb{R}^3$ with mean curvature $H$ satisfying \eqref{EL} and denote by $\Omega\subset\mathbb{R}^3$ the open domain bounded by $\Sigma$. The divergence of  the vector field $\xi=(\eta f(z)+\varpi)e_3$ defined in $\mathbb{R}^3$ is $\eta f'(z)$. Hence the divergence theorem implies
\begin{equation}\label{eq-la1}
\eta\int_\Omega f'(z)\,dV=\int_\Sigma \left(\eta f(z)+\varpi\right)\langle\nu,e_3\rangle\,d\Sigma\,,
\end{equation}
where $\nu$ is the unit normal vector field on $\Sigma$ pointing outwards $\Omega$. The left-hand side of \eqref{eq-la1} is nonzero by the assumption that $f(z)$ is increasing almost everywhere (respectively decreasing) and $\eta\neq 0$. On the other hand, computing the Laplacian with respect to the metric induced on $\Sigma$ of the height function $z=\langle X,e_3\rangle$,   we know
$$\sigma\Delta\langle X,e_3\rangle=2\sigma H\langle\nu,e_3\rangle=\left(\eta f(z)+\varpi\right)\langle\nu,e_3\rangle\,,$$
where we have used the Euler-Lagrange equation \eqref{EL}. Since $\Sigma$ is a closed surface, and using \eqref{eq-la1}, we get
$$0=\sigma \int_\Sigma\Delta\langle X,e_3\rangle\, d\Sigma=\int_\Sigma \left(\eta f(z)+\varpi\right)\langle \nu,e_3\rangle\,d\Sigma=\eta \int_\Omega f'(z)\, d V\neq 0\,,$$ 
obtaining a contradiction. The last statement is immediate.  
\end{proof}

\section{Critical cylinders}\label{3}

Among surfaces critical for the energy $E[\Sigma]$, a special type which merits further investigation are cylinders. As explained in the introduction, these surfaces extend the classical problem of the \emph{lintearia} ($f(z)=z$) to the consideration of more sophisticated vertical potential energies.

The purpose of this section is to study equilibrium immersions invariant under translations. If the surface is invariant in the direction of a unit vector $v\in \mathbb{R}^3$, then $\Sigma$ can be parameterized as 
\begin{equation}\label{param0}
X(s,t)=\gamma(s)+t v\,,
\end{equation}
where $\gamma(s)$ is a planar curve called the generating curve and contained in an orthogonal plane to $v$. The parameter $s\in I\subset\mathbb{R}$ denotes the arc length parameter of $\gamma$. These surfaces are referred as to {\it right cylinders} shaped on the curve $\gamma$.

Denote by $T(s)=\gamma'(s)$ the unit tangent vector field along the planar curve $\gamma(s)$, where $\left(\,\right)'$ represents the derivative with respect to the arc length parameter $s$, and define the unit normal vector field $N(s)$ along $\gamma(s)$ to be the counter-clockwise rotation of $T(s)$ through an angle $\pi/2$ in the plane where $\gamma(s)$ lies, i.e., $N(s)=JT(s)$. In this setting, the Frenet-Serret equation
$$T'(s)=\kappa(s)N(s)\,,$$
defines the (signed) curvature $\kappa(s)$ of $\gamma(s)$.

The unit normal $\nu$ to the right cylinder $X$ parameterized as \eqref{param0}  is $\nu(s,t)=N(s)$ and the mean curvature  is  $H(s,t)= \kappa(s)/2$. Then the equilibrium condition \eqref{EL} is
\begin{equation}\label{ELk}
\sigma\kappa(s)=\eta f\left(z(s,t)\right)+\varpi\,.
\end{equation}
Here $z(s,t)$ is the $z$-coordinate of $X(s,t)$,  
$$z(s,t)=\langle X(s,t),e_3\rangle=\langle \gamma(s),e_3\rangle+t\langle v,e_3\rangle\,.$$
Differentiating \eqref{ELk} with respect to $t$, we have $0=\eta f'(z)\langle v,e_3\rangle$, hence $\langle v,e_3\rangle=0$ (recall that we are assuming $\eta\neq 0$ and $f(z)$ nonconstant). Therefore the rulings of $\Sigma$ are orthogonal to the vertical direction $e_3$ which, after a rotation about the direction $e_3$ (this does not carry any change on $f(z)$), we may assume parallel to $e_1=(1,0,0)$. Consequently, $\Sigma$ can be parameterized as 
\begin{equation}\label{param}
X(s,t)=\gamma(s)+t e_1\,,
\end{equation}
where now $\gamma(s)=\left(0,y(s),z(s)\right)$ is a planar curve contained in the $yz$-plane.

From now on we will consider planar curves $\gamma\colon I\subset\r\rightarrow\r^2$ and we will denote $\gamma(s)=(y(s),z(s))$ the coordinate functions of $\gamma$. We will also restrict ourselves to the cases $f(z)=z^n$ for any real number $n\not=0$.  Then \eqref{ELk} reads
\begin{equation}\label{zn0}
\sigma \kappa=\eta\, z^n+\varpi\,,
\end{equation}
where $z^n=z(s)^n$. We show that  the energy parameters $\sigma$ and $\eta$ can be fixed after reparameterizations and dilations of $\gamma$. Indeed, if $\gamma$ satisfies \eqref{zn0}, reversing the orientation of $\gamma$, the curve $\widetilde{\gamma}(s)=\gamma(-s)$ satisfies \eqref{zn0} by reversing the signs of $\eta$ and ${\varpi}$. Similarly, the dilation $\widehat{\gamma}(rs)=r\gamma(s)$ rescales the parameters $\eta$ and ${\varpi}$ by $\eta r^{-n-1}$ and $\varpi r^{-1}$, respectively. After these simplifications, throughout this paper we will use the following definition.
 
\begin{definition} An arc length parameterized planar curve $\gamma(s)=(y(s),z(s))$ is a $n$-elastic curve, if its curvature $\kappa(s)$ satisfies
 \begin{equation}\label{zn}
\kappa(s)=z(s)^n+\mu\,,
\end{equation}
for some real constant $\mu$.
\end{definition}
With this definition, which fixes some of the energy parameters, and the choice of the function $f(z)=z^n$ the energy $E[\Sigma]$ reads 
\[E[\Sigma]= \mathcal{A}[\Sigma]+ \int_\Omega z^n\,dV+\mu\mathcal{V}[\Omega].\]
Equation \eqref{zn} can be seen as a prescribed curvature equation for planar curves. In the present case, the curvature depends on the distance to a fixed straight line. 

First, we analyze the curves with constant curvature that are solutions of \eqref{zn}. If $\kappa\not=0$ then $\gamma$ is a circle, which clearly does not satisfy \eqref{zn}. On the other hand, if $\kappa=0$ then $\gamma$ is a straight line and by \eqref{zn} it must be a horizontal line. We focus on this case in the following result.  

\begin{proposition}\label{lines} Let $\gamma$ be a curve whose constant curvature is a solution of \eqref{zn}. Then $\gamma$ is a horizontal straight line. Moreover: 
\begin{enumerate}
\item Case $\mu>0$. Then $\gamma$ is $z=\sqrt[n]{-\mu}$ and it exists if and only if $n$ is odd.
\item Case $\mu=0$. Then $\gamma$ is $z=0$.  
\item Case $\mu<0$. If $n$ is not even then $\gamma$ is $z=\sqrt[n]{-\mu}$ and, if $n$ is even $\gamma$ is one of the two straight lines $z=\pm\sqrt[n]{-\mu}$.
\end{enumerate}
\end{proposition}

For those planar curves with nonconstant curvature which are solutions of \eqref{zn} we will locally characterize them as planar critical curves for a curvature energy. We briefly recall here the general theory for curvature energies. Consider a general curvature energy functional
\begin{equation}\label{energycurve}
\mathbf{\Theta}[C]=\int_C P(\kappa)\,ds\,,
\end{equation}
where $P(\kappa)$ is a smooth function defined in an adequate domain. By standard computations involving integrating by parts we can calculate the first variation formula associated to $\mathbf{\Theta}$ (see details in \cite{P}) obtaining the Euler-Lagrange equation  
\begin{equation}\label{Jprima}
\left(\left(\kappa\dot{P}-P\right)T+\dot{P}_s\,N\right)'=0\,,
\end{equation}
where $\dot{P}(\kappa)$ denotes the derivative of $P(\kappa)$ with respect to $\kappa$. Curves whose curvature $\kappa(s)$ is a solution of \eqref{Jprima} are called {\it critical curves} throughout the paper, regardless of the boundary conditions. We introduce the vector field
\begin{equation}\label{J}
\mathcal{J}=\left(\kappa\dot{P}-P\right)T+\dot{P}_s\,N\,.
\end{equation}
From \eqref{Jprima} it is then clear that along a critical curve $C$, the vector field $\mathcal{J}$ is constant and, therefore, $\|  \mathcal{J}\|^2=d$ for some positive real constant $d$, represents a first integral of the Euler-Lagrange equation. Expanding it, we obtain
\begin{equation}\label{pss}
\dot{P}_s^2+(\kappa\dot{P}-P)^2=d\,.
\end{equation}
We next prove a result characterizing critical curves for $\mathbf{\Theta}$ in terms of their parameterization. Here we will use Killing vector fields along curves in the sense of Langer and Singer (\cite{LS}).

\begin{proposition}\label{coord} Assume that $\dot{P}_s\neq 0$. An arc length parameterized planar curve $C(s)$ with curvature $\kappa(s)$ is critical for $\mathbf{\Theta}$ if and only if there is a coordinate system such that $C(s)=\left(C_1(s),C_2(s)\right)$ and
\begin{equation}\label{coord1}
C_2(s)=\frac{1}{\sqrt{d\,}}\dot{P}\left(\kappa(s)\right)
\end{equation}
for any constant $d>0$.
\end{proposition}
\begin{proof} Let $C(s)=\left(C_1(s),C_2(s)\right)$ be a planar critical curve for $\mathbf{\Theta}$. The vector field $\mathcal{J}$ is a Killing vector field along $C$ which can be uniquely extended to a Killing vector field on the whole space $\mathbb{R}^2$. Since $\lVert \mathcal{J}\rVert^2=d$ is constant,   the extension of $\mathcal{J}$ to $\mathbb{R}^2$ (also denoted by $\mathcal{J}$) is a translational Killing vector field. After a rigid motion if necessary, we can assume that $\mathcal{J}=(\sqrt{d\,},0)$. Next, from \eqref{Jprima} we obtain
$$\kappa\dot{P}-P=\langle T, \mathcal{J}\rangle=\langle\left(C_1',C_2'\right),(\sqrt{d\,},0)\rangle=\sqrt{d}\,C_1'\,,$$
so that
$$C_1'(s)=\frac{1}{\sqrt{d\,}}\left(\kappa\dot{P}-P\right).$$
Finally, we use that $C(s)$ is parameterized by arc length and that \eqref{pss} is satisfied to conclude that
$$C_2'(s)=\frac{1}{\sqrt{d\,}}\dot{P}_s\,.$$
After integrating and translating, if necessary, we obtain the forward implication.

For the reverse implication, assume that $C(s)=\left(C_1(s),C_2(s)\right)$ is a planar curve parameterized by arc length and such that \eqref{coord1} holds for some $d>0$, where $\kappa(s)$ denotes its curvature. We consider the arc length parameterized curve
$$\widetilde{C}(s)=\frac{1}{\sqrt{d\,}}\left(\int\left(\kappa\dot{P}-P\right)ds,\dot{P}\right),$$
which is critical for $\mathbf{\Theta}$ since it satisfies \eqref{Jprima}. The curvature of $\widetilde{C}(s)$ locally coincides with the curvature of $C(s)$, $\kappa(s)$. Therefore, by the Fundamental Theorem of Planar Curves, $C(s)=\widetilde{C}(s)$, after a rigid motion. Consequently, $C(s)=\widetilde{C}(s)$ is critical for $\mathbf{\Theta}$.\end{proof}

Observe that the restriction $\dot{P}_s\neq 0$ is quite natural for our purposes. Indeed, since we are assuming that $\kappa$ is not constant, $\dot{P}_s=0$ if and only if $P(\kappa)=a\kappa+b$, for real constants $a$ and $b$. If $b=0$, $\mathbf{\Theta}$ represents the total curvature whose associated Euler-Lagrange equation is an identity. On the contrary, if $b\neq 0$, critical curves for $\mathbf{\Theta}$ are straight lines ($\kappa=0$), which are out of our consideration.

Using Proposition \ref{coord}, we prove the main result of this section.

\begin{theorem}\label{crit} Let $\gamma$ be a planar curve with nonconstant curvature. Then $\gamma$ is a $n$-elastic curve if and only if it satisfies the Euler-Lagrange equation associated to the curvature energies:
\begin{eqnarray*}
\mbox{Case $n\not=-1$:} && \mathbf{\Theta}[\gamma]=\int_\gamma \left(\left(\kappa-\mu\right)^p+\lambda\right)ds\,,\\
\mbox{Case $n=-1$:}& & \mathbf{\widetilde{\Theta}}[\gamma]=\int_\gamma \left(\log\left(\kappa-\mu\right)+\lambda\right)ds\,,
\end{eqnarray*}
where $p=(n+1)/n$ and $\lambda\in\mathbb{R}$. (If $p\in\r\setminus\n$ or $n=-1$, above energies must be understood as acting on spaces of curves satisfying $\kappa>\mu$.) 
\end{theorem}
\begin{proof} For the forward implication, from \eqref{zn} we have that $z(s)=\left(\kappa(s)-\mu\right)^{1/n}$. If we take $\dot{P}(\kappa)=(\kappa-\mu)^{1/n}$, the condition \eqref{coord1} of Proposition \ref{coord} is satisfied, concluding after integrating that $\gamma$ is critical for the energies of the statement.

For the converse, we consider first the case $n\neq -1$. Let $\gamma(s)=\left(\gamma_1(s),\gamma_2(s)\right)$ be an arc length parameterized planar critical curve for $\mathbf{\Theta}$. From Proposition \ref{coord} we know that there exists a coordinate system in which $\gamma$ can be parameterized as
\begin{equation}\label{pc1}
\gamma(s)=\frac{1}{\sqrt{d\,}}\left(\int \left(\left(\kappa-\mu\right)^{p-1}\left((p-1)\kappa+\mu\right)-\lambda\right)ds,p\left(\kappa-\mu\right)^{p-1}\right),
\end{equation}
for some constant $d>0$. After a rigid motion, reflection and dilation, if necessary, we may assume that $\gamma_2(s)=z(s)=(\kappa-\mu)^{p-1}$. Then, since $p=(n+1)/n$, \eqref{zn} is satisfied. 

A similar argument for the case $n=-1$ gives that a critical curve $\gamma$ for $\mathbf{\widetilde{\Theta}}$ can be parameterized as
\begin{equation}\label{pc2}
\gamma(s)=\frac{1}{\sqrt{d\,}}\left(\int\left(\frac{\kappa}{\kappa-\mu}-\log(\kappa-\mu)-\lambda\right)ds,\frac{1}{\kappa-\mu}\right),
\end{equation}
for some constant $d>0$. As before, we may assume $z(s)=(\kappa-\mu)^{-1}$ so that \eqref{zn} holds, obtaining the result.
\end{proof}

We give some observations for particular choices of the constants $\lambda$, $\mu$ and $p$ in Theorem \ref{crit}:
\begin{enumerate}
\item Case $\lambda=\mu=0$ and $p=2$. Here we recover the classical bending energy of curves, which corresponds with right cylinders that are critical for $E[\Sigma]$ for $n=1$ and $\varpi=0$. This relation was pointed out in \cite{KP2}, where the authors acknowledge the comments of Prof. O. J. Garay. 
\item  Case  $\lambda=\mu=0$ and $p=1/2$. This energy was studied by Blaschke in 1930 (\cite{Bl}) obtaining that critical curves are catenaries. 
\item More generally, the case $\lambda=0$, $p=1/2$ and $\mu\in\mathbb{R}$ gives rise to roulettes of conic foci (\cite{AGP}). These cases correspond with the choice $n=-2$ in \eqref{zn}. Planar curves satisfying \eqref{zn} for $n=-2$ have been studied in \cite{I,M}, although this variational characterization was not described. 
\item Case $\lambda=0$ and $\mu, p\in\r$. This case was used in the characterization of rotational linear Weingarten surfaces in \cite{LP}, where the authors gave a full classification of the critical curves.  
\end{enumerate}

From the proof of Theorem \ref{crit}, curves satisfying \eqref{zn} can be parameterized, up to rescaling and change of orientation, as \eqref{pc1} which combined with \eqref{zn} yields
\begin{equation}\label{pc11}
\gamma(s)=\left(-\int_0^{s}\left(\frac{1}{n+1}z(s)^{n+1}+\mu z(s)-\frac{n}{n+1}\lambda\right)ds,z(s)\right)
\end{equation}
when $n\neq -1$, while for the case $n=-1$ the parameterization of $\gamma(s)$ is \eqref{pc2} and, once again, combining it with \eqref{zn}, 
\begin{equation}\label{pc22}
\gamma(s)=\left(-\int_0^{s}\left(\mu z(s)+\log z(s)+1-\lambda\right)ds,z(s)\right).
\end{equation}
We highlight here that these last two parameterizations are given in terms of just one quadrature. In fact, we can combine \eqref{pss} with the energies given in Theorem \ref{crit}, to make a change of variable in the integral of the parameterizations and, hence, obtaining locally a graph which can be recovered after just one quadrature.

Another observation of this variational approach and parameterizations is that curves satisfying \eqref{zn} are theoretically characterized as solutions of a first order ordinary differential equation. Indeed, we have the following result.

\begin{proposition}\label{integralequation} Let $\gamma(s)=\left(y(s),z(s)\right)$ be a curve parameterized by arc length. Then $\gamma$ is a $n$-elastic curve if and only if $z(s)$ satisfies the first order ordinary differential equation:
\begin{eqnarray}
\mbox{Case $n\not=-1$:} && \frac{1}{n+1}z^{n+1}+\mu z+\epsilon \sqrt{1-\left(z'(s)\right)^2}=c\,,\label{ode1}\\
\mbox{Case $n=-1$:}& & \log z+\mu z+\epsilon \sqrt{1-\left(z'(s)\right)^2}=c\,.\label{ode2}
\end{eqnarray}
In both cases, $\epsilon=\pm 1$ and $c$ is a suitable real constant.
\end{proposition}
\begin{proof} As mentioned above, if $n\neq -1$ a $n$-elastic curve is   parameterized by  \eqref{pc11}.  From this equation we obtain that $$y'(s)=-\frac{1}{n+1}z(s)^{n+1}-\mu z(s)+\frac{n}{n+1}\lambda\,,$$
which combined with the arc length condition $y'(s)^2+z'(s)^2=1$ gives \eqref{ode1}, where $c=n\lambda/(n+1)$.

Similarly, if $n=-1$ we use \eqref{pc22} together with $y'(s)=\epsilon\sqrt{1-\left(z'(s)\right)^2}$ to conclude the result, for $c=\lambda-1$.
\end{proof}

Observe that the horizontal straight lines solution of \eqref{zn} (see Proposition \ref{lines}) can be included in the statement of this proposition.

From Proposition \ref{integralequation} we directly conclude some geometric properties of $n$-elastic curves. First, we observe that in the cases where $n\leq-1$ the curve cannot meet the line of equation $z=0$, since if that happens equation \eqref{ode1} (or equation \eqref{ode2}) is not well defined.

\begin{corollary}\label{n1} If $n\leq-1$, a $n$-elastic curve cannot meet the $y$-line. 
\end{corollary}

Second, by letting $z\rightarrow\infty$ in \eqref{ode1} and \eqref{ode2}, we give conditions so that the function $z(s)$ is bounded.

\begin{corollary} \label{n2} If either $\mu\not=0$ or  $\mu=0$ and $n\geq -1$, the function $z(s)$ is bounded.
\end{corollary}

From \eqref{ode1} and \eqref{ode2}, we consider the function $F$ of two variables defined by 
\begin{equation}\label{F}
F(u,v)=\left\{\begin{array}{ll}
\frac{1}{n+1}u^{n+1}+\mu u+\epsilon \sqrt{1-v^2},&\quad\quad n\not=-1,\\
& \\
\log u+\mu u+\epsilon \sqrt{1-v^2}, &\quad\quad n=-1.\end{array}\right.
\end{equation}
Equations \eqref{ode1} and \eqref{ode2} tell us that for a $n$-elastic curve, the pair $(z,z')$ belongs to a level curve of the function $F(u,v)$ with $u>0$ and $1-v^2\geq 0$. This function $F$ can be understood as a Morse function and, we can study its orbit space geometry to deduce the behavior of all solutions of \eqref{zn}. For example, the equilibrium points are the zeroes of $\nabla F$. In particular, $z'=0$ and the solutions corresponding to the equilibrium points are horizontal straight lines, as described in Proposition \ref{lines}.

\section{Geometric properties of  $n$-elastic curves}\label{4}

In this section we study properties regarding symmetries of $n$-elastic curves. Because of the presence of $\sqrt{1-v^2}$ in the expression \eqref{F} of $F$, we will introduce here a different approach.

Let $\gamma(s)=\left(y(s),z(s)\right)$, $s\in I\subset\r$, be a $n$-elastic curve which is parameterized by the arc length. Then $y'(s)=\cos\theta(s)$ and $z'(s)=\sin\theta(s)$, where $\theta$ is the angle between the tangent vector of $\gamma$ and the positive part of the $y$-axis. Then equation \eqref{zn} is equivalent to  
\begin{equation}\label{eq1}
	\begin{split}
		y'(s)&=\cos\theta(s)\\
		z'(s)&=\sin\theta(s)\\
		\theta'(s)&=z(s)^n+\mu\,.
	\end{split}
\end{equation}
In what follows, we will assume that the domain $I=(s_{-},s_{+})\subset\r$ of $\gamma$ is the maximal interval for which the solution of the system (\ref{eq1}) exists. For example, if $n\in\z$,   the maximal domain of \eqref{eq1} is $\r$. On the contrary, assume that  $s_{+}<\infty$. Since $z(s)^n$ can take any real value if $n>0$ and that $\gamma$ cannot meet the line $z=0$ if $n<0$ (Corollary \ref{n1}), then necessarily $\theta$ blows up at $s_{+}$, so  $\lim_{s\rightarrow s_+}\theta(s)=\infty$. By the second equation of \eqref{eq1}, $z'$ is bounded close to $s_{+}$ and this implies that $z$ is bounded near $s_{+}$. Now by the third equation of \eqref{eq1}, $\theta'$ is bounded near to $s_{+}$, a contradiction. A similar argument works for the case where $s_{-}>-\infty$.

We now impose the initial conditions for \eqref{eq1}. Since \eqref{eq1} is invariant by translations in the $y$-direction, we can assume $y(0)=0$. Fixing the initial value $\theta(0)$ for the function $\theta(s)$ is equivalent to fixing the initial velocity $\gamma'(0)$. For the classification of all solutions of \eqref{eq1}, we need to assume all initial conditions 
\begin{equation}\label{eq2}
y(0)=0\,,\quad\quad\quad z(0)=z_0\,,\quad\quad\quad \theta(0)=\theta_0\,, 
\end{equation}
with $\theta_0\in [0,2\pi)$ and suitable real constant $z_0$. For any $n\not\in\z$, the function $z^n$ in \eqref{eq1} is only defined for positive values of $z$, so that $z_0>0$. In case that $n\in\z$, then $z$ can also be negative and so $z_0\in\r$ if $n>0$ and $z_0\in\r\setminus\{0\}$ if $n<0$. 

For the case where the power $n$ is an integer, we prove some symmetries of the solutions. Exactly, if $n$ is even, we see that it suffices to consider nonnegative values $z_0$ in \eqref{eq2}, and if $n$ is odd, after a change on the sign of $\mu$, if necessary, it also suffices to consider $z_0\geq 0$.

\begin{proposition}[Horizontal symmetry]\label{pr6} Let $\gamma(s)=\left(y(s),z(s)\right)$ be a solution of the system of differential equations \eqref{eq1}--\eqref{eq2} with $n\in\z$ and denote by $\bar{\gamma}$ the curve obtained after applying a symmetry with respect to the $y$-line to $\gamma$. Then:
\begin{enumerate}
\item If $n$ is even, $\bar{\gamma}$ is a $n$-elastic curve for the same constant $\mu$. 
\item If $n$ is odd, $\bar{\gamma}$ is a $n$-elastic curve for the constant $-\mu$.  
\end{enumerate}
\end{proposition}
\begin{proof}  
We consider first the case $n$ even. Define $\bar{y}(s)=y(-s)$, $\bar{z}(s)=-z(-s)$ and $\bar{\theta}(s)=\pi-\theta(-s)$. Then it is immediate that $\{\bar{y}(s),\bar{z}(s),\bar{\theta}(s)\}$ satisfies \eqref{eq1} for the initial solution $\{0,-z_0,\pi-\theta_0\}$. Here the fact that $n$ is even is essential since
$$\bar{\theta}'(s)=\theta'(-s)=z(-s)^n+\mu=(-\bar{z}(-s))^n+\mu=\bar{z}(s)^n+\mu\,.$$
If $n$ is odd, define $\bar{y}(s)=y(s)$, $\bar{z}(s)=-z(s)$ and $\bar{\theta}(s)=-\theta(s)$. Then it is immediate that $\{\bar{y}(s),\bar{z}(s),\bar{\theta}(s)\}$ satisfies \eqref{eq1} reversing the sign of $\mu$ and for the initial solution $\{0,-z_0,-\theta_0\}$. In this case,
$$\bar{\theta}'(s)=-\theta'(s)=-z(s)^n-\mu=-(-\bar{z}(s))^n-\mu=\bar{z}(s)^n-\mu\,,$$
where we have used the fact that $n$ is odd in an essential way.
\end{proof} 

In a second step of our program of classification, we study under what circumstances $n$-elastic curves have a vertical symmetry. We will prove that this always occurs and, consequently, it is enough to consider the cases $\theta_0=0$ and $\theta_0=\pi$ in the initial conditions \eqref{eq2}.

\begin{proposition}[Vertical symmetry]\label{pr1}
Any $n$-elastic curve $\gamma$ has a point $s_0$ where the tangent vector $\gamma'(s_0)$ is horizontal. Furthermore, the trace of $\gamma$ is symmetric about the vertical line of equation $y=y(s_0)$.
\end{proposition} 
\begin{proof} The existence of $s_0$ is proved by contradiction. Let $\gamma(s)=(y(s),z(s))$ be a solution of \eqref{eq1}--\eqref{eq2} and let $(s_{-},s_{+})$ be its maximal domain. If $\gamma'(s)$ is never horizontal, this implies that, up to an integer multiple of $2\pi$,  the image of $\theta(s)$ is included in $(0,\pi)$ or in $(\pi,2\pi)$. We prove the case $\theta_0\in (0,\pi)$, which also implies the case $\theta_0\in (\pi,2\pi)$ after applying Proposition \ref{pr6}. 

If $n\not\in\n$, then $z_0>0$. If $n\in \z$, we apply Proposition \ref{pr6} together with a change on the sign of $\mu$, if necessary, and suppose $z_0>0$ if $n\in \z^{-}$ or $z_0\geq 0$ if $n\in\n$.  

Since $\theta(s_-,s_+)\subset (0,\pi)$, from $z'(s)=\sin\theta(s)$, we have that $z(s)$ is an increasing function. We first prove that $s_{+}=\infty$. Otherwise, and using that $z(s)>z_0$ for $s>0$, then $z(s)$ cannot go to $0$  so this implies that  $\theta'(s)\rightarrow\infty$ as $s\rightarrow s_+$. Since $\theta'=z^n+\mu$ then $z^n\rightarrow\infty$. Using again that $z(s)>z_0$ for $s>0$, we obtain that $n>0$ and $z(s)\rightarrow \infty$ as $s\rightarrow s_{+}$, which contradicts Corollary \ref{n2}. Once proved that $s_+=\infty$,  notice that if $s_{-}>-\infty$, then as $s\rightarrow s_{-}$ we have either $\theta'(s)\rightarrow\infty$ or $n\not\in\n$ with $z(s)\rightarrow 0$. We distinguish three cases depending on the sign of $\mu$:
\begin{enumerate}
\item Case $\mu> 0$. Then $\theta'(s)\geq \mu$, so $\theta$ is an increasing function. Next, we use that $s_+=\infty$, to get $\lim_{s\rightarrow\infty}\theta(s)=+\infty$, a contradiction, because the rank of $\theta$ is bounded. 
\item Case $\mu=0$. Again $\theta$ is increasing and bounded from above, so $\lim_{s\rightarrow\infty}\theta'(s)=0$. This cannot occur if  $n>0$ because $\theta'(s)\geq z_0^n$ for $s>0$: in case that $n\in\n$ and $z_0=0$, using that $z$ is increasing, then $\theta'(s)=z(s)^n\geq z(\delta)^n>0$ for some fixed $\delta>0$, obtaining a contradiction again. Therefore, $n<0$. Then $\lim_{s\rightarrow\infty}\theta'(s)=0$ implies $\lim_{s\rightarrow\infty}z(s)=\infty$, hence $n<-1$ by Corollary \ref{n2}. If $s_{-}=-\infty$, since $\theta$ is bounded and increasing, then $\lim_{s\rightarrow s_{-}}\theta'(s)=0$. This implies $z(s)\rightarrow \infty$, which is not possible because $z(s)$ is increasing. Thus $s_{-}>-\infty$ and $\lim_{s\rightarrow s_{-}}z(s)=0$, obtaining a contradiction again from  Corollary \ref{n1}.
\item Case $\mu<0$. From Corollary \ref{n2}, $z(s)$ is a function bounded from above. Because $z(s)$ is increasing, let $z_1=\lim_{s\rightarrow\infty}z(s)>0$. Furthermore, $\lim_{s\rightarrow\infty}z'(s)=0$, so $\lim_{s\rightarrow\infty}\theta(s)$ is $0$ or $\pi$. Letting $s\rightarrow\infty$, we have  $\lim_{s\rightarrow\infty}\theta'(s)=z_1^n+\mu$. Since $\theta$ is bounded, the above limit must be $0$ so $z_1^n+\mu=0$. Now the function $\theta'$ is monotonic at infinity because $\theta''(s)=nz(s)^{n-1}\rightarrow nz_1^{n-1}\not=0$. This is a contradiction because $\theta'$ is bounded at infinity.  
\end{enumerate}
Once proved the existence of $s_0$, we see that $\gamma $ is symmetric about the vertical line through $\gamma(s_0)$. Since $\theta(s_0)=m\pi$ for $m\in\mathbb{Z}$, then $z'(s_0)=\sin\theta(s_0)=0$. The functions 
\[\bar{y}(s)= 2y(s_0)-y(2s_0-s)\,,\quad \bar{z}(s)=  z(2s_0-s)\,,\quad \bar{\theta}(s)=2m\pi-\theta(2s_0-s)\,\]
 satisfy the same equations (\ref{eq1}) with the same initial conditions at $s=s_0$ that $\{y(s),z(s),\theta(s)\}$. The proof follows from the uniqueness of solution of ordinary differential equations.
\end{proof} 

In conclusion, after suitable symmetries described in Propositions \ref{pr6} and \ref{pr1},  we can restrict the initial conditions \eqref{eq2} to be   $\theta_0=0$ or $\theta_0=\pi$ and $z_0\geq 0$ (strictly positive if $n\not\in\n$).

Observe that integrating \eqref{eq1} we recover the curve $\gamma(s)$ after two quadratures. A parameterization of $\gamma(s)$ using only one quadrature was given in previous section. Of course, both parameterizations are related and initial conditions coincide. Clearly $y(0)=0$ is satisfied in \eqref{pc11} and \eqref{pc22} due to the choice of the limits of integration. Moreover, the other two initial conditions can be described in terms of the Lagrange multiplier $\lambda$ restricting the length of the curve. In fact, assume that $n\neq -1$, differentiating $y(s)$ in \eqref{pc11} once and combining it with \eqref{eq1}, we obtain
$$y'(s)=-\frac{z(s)^{n+1}}{n+1}-\mu z(s)+\frac{n\lambda}{n+1}=\cos\theta(s)\,.$$
In particular, evaluating this at the initial value $s=0$ and using \eqref{eq2}, we obtain an expression of $\lambda$ in terms of $z_0$ and $\theta_0$, namely,
\begin{equation}\label{lambda}
\lambda=\frac{n+1}{n}\cos\theta_0+\frac{z_0^{n+1}}{n}+\frac{n+1}{n}\mu z_0\,.
\end{equation}
In a similar way, the Lagrange multiplier $\lambda$ for the case $n=-1$ can be described in terms of the initial conditions \eqref{eq2} as $\lambda=1+\mu z_0+\log z_0+\cos\theta_0$. 

\begin{remark}\label{rem} We describe here the initial conditions of a couple of relevant families of $n$-elastic curves:
\begin{enumerate}
\item The case $n=1$ and $\mu=0$ corresponds with the classical elastic curves. From \eqref{lambda} it must be the case that $$\lambda=2\cos\theta_0+z_0^2$$
holds for the initial conditions \eqref{eq2}. For arbitrary initial conditions, we recover all the cases of Euler's classification (\cite{Euler}). Compare with Figure \ref{figtipos}.
\item Using the variational description of \cite{AGP}, Delaunay curves appear when $n=-2$ and $\lambda=0$. Delaunay curves are roulettes of conic foci and the profile curves of rotational constant mean curvature surfaces (\cite{Delaunay}). Using \eqref{lambda}, the following relation must hold:  
$$1=z_0\cos\theta_0+\mu z_0^2.$$ 
In particular, we get undularies ($\mu<0$),   catenaries ($\mu=0$) and nodaries ($\mu>0$). See Figure \ref{figtipos2} for the family of catenary-like curves.
\end{enumerate}
\end{remark}

\section{Existence of closed $n$-elastic curves}\label{seccer}

As it was pointed out in the introduction, an interesting problem is to determine wether or not there exist closed $n$-elastic curves. Looking in the classical theory of elastic curves ($n=1$ and $\mu=0$), it is known the existence of a non-simple planar closed elastic curve, the so-called Bernoulli's lemniscate or the elastic figure-eight. This curve is also the only nontrivial planar closed curve for $n=1$ and arbitrary $\mu$, since the total curvature is constant on a regular homotopy class of planar curves. However, we will see that, apart from these types of curves which we called {\it pseudo-lemniscate} (Figure \ref{figtipos}, (d)), for $n\in\n$ even the family of closed $n$-elastic curves is richer and also includes simple closed curves. See Figure \ref{cerradas}.

In order to have closed curves, a necessary but not sufficient condition is to have curves with periodic curvature or, due to \eqref{zn}, equivalently, curves whose $z$-component is periodic. In the following result we give sufficient conditions for that to happen.

\begin{proposition}\label{pr2}
Suppose that $\gamma(s)=\left(y(s),z(s)\right)$ is a solution of (\ref{eq1})--(\ref{eq2}) such that the image of the function $\theta(s)$ contains a closed interval of type $[m\pi,(m+1)\pi]$, $m\in\z$. Then the curvature of $\gamma$ is a periodic function and, consequently, the solution is defined in $\r$. Moreover, $\gamma$ is either a closed curve or invariant by a discrete group of horizontal translations.  
\end{proposition}
\begin{proof} Since the solutions of \eqref{eq1}--\eqref{eq2} are the same if we change $\theta_0$ in \eqref{eq2} by $\theta_0+2m\pi$, $m\in\z$, we may suppose that the rank of $\theta$ contains the interval $[0,\pi]$. Assume that $\theta(0)=\theta_0=0$ and let $s_0>0$ be the first point such that   $\theta(s_0)=\pi$. By Proposition \ref{pr1}, the solution $\gamma(s)=\left(y(s),z(s)\right)$ is symmetric about the vertical lines $y=0$ and $y=y(s_0)$. Therefore,  
\begin{align*}
y(s)&= 2y(s_0)-y(2s_0-s),\\
z(s)&=  z(2s_0-s),\\
\theta(s)&=2\pi-\theta(2s_0-s).
\end{align*}
Consider the value $\rho=2s_0$. Then $\gamma(\rho)=(2y(s_0),z_0)$  and $\theta(\rho)=2\pi$. Define 
\[\bar{y}(s)= y(s+\rho)-y(\rho),\quad\quad \bar{z}(s)=  z(s+\rho),\quad\quad \bar{\theta}(s)=\theta(s+\rho)-2\pi.\]
These functions satisfy \eqref{eq1} with initial conditions $\{0,z_0,0\}$. By uniqueness, these solutions coincide with $\{y(s),z(s),\theta(s)\}$. Consequently,
\begin{align*}
y(s+\rho)&= y(s)+y(\rho),\\
 z(s+\rho)&=  z(s),\\
  \theta(s+\rho)&=\theta(s)+2\pi.
  \end{align*}
The first two identities imply that $\gamma(s+\rho)=\gamma(s)+(y(\rho),0)$. If $y(\rho)=0$, then $\gamma$ is closed. Otherwise, $\gamma$ is invariant by the action of the group of horizontal translations generated by the vector $(y(\rho),0)$. The second identity implies that $z=z(s)$ is periodic, and by the third equation in \eqref{eq1}, the same occurs for the curvature function $\kappa(s)=\theta'(s)$. 
\end{proof}

The first class of closed curves are the pseudo-lemniscates, which have the shape of a figure-eight and play the role of Bernoulli's lemniscate (Figure \ref{figtipos}, (d)). These closed curves appear whenever suitable closure conditions, obtained from \eqref{pc11} and \eqref{pc22}, are satisfied. Indeed, as mentioned in the proof of Proposition \ref{pr2}, a curve with periodic curvature $\kappa(s)$, of period $\rho$, will be closed if and only if $y(\rho)=0$. It is then clear that this closure condition can be rewritten as
\begin{equation}\label{closure1}
\int_0^\rho \left(\frac{1}{n+1}z(s)^{n+1}+\mu z(s)-\frac{n}{n+1}\lambda\right)ds=0\,,
\end{equation}
when $n\neq -1$, while for the case $n=-1$, the closure condition reads
\begin{equation}\label{closure2}
\int_0^\rho\left(\mu z(s)+\log z(s)+1-\lambda\right)ds=0\,.
\end{equation}
For each case, an analysis of above integrals, which depend on $\lambda$, may be used to show the existence of pseudo-lemniscates. In Figure \ref{figtipos}, we have shown a complete deformation of $n$-elastic curves obtained varying the Lagrange multiplier $\lambda$ (equivalently, the initial conditions $z_0$ and $\theta_0$), which imitate the cases of classical elastic curves.

We next give a result for the nonexistence of simple closed $n$-elastic curves. The idea behind is similar to that of Proposition \ref{pr-closed}. 

\begin{proposition} \label{pr3} If $n$ is not an even natural number, simple closed $n$-elastic curves do not exist.
\end{proposition}
\begin{proof} By contradiction, let us assume that $\gamma$ is a simple closed $n$-elastic curve. On $\r^2$ we consider the vector field $\xi=(z^n+\mu)e_2$, where $e_2=(0,1)$. Then its divergence is  $\mbox{div}\,\xi=nz^{n-1}$. If $\Omega\subset\r^2$ is the domain bounded by $\gamma$, then the Frenet normal $N(s)$ is outward pointing and by the divergence theorem
\begin{align*}
\int_\Omega nz^{n-1}\,dA&=\int_\gamma (z^n+\mu)\langle N(s),e_2\rangle\, ds=\int_\gamma\langle\kappa(s)N(s),e_2\rangle\, ds\\
&=\int_\gamma \langle T'(s),e_2\rangle\,ds=\int_\gamma\frac{d}{ds}\langle T(s),e_2\rangle\,ds=0\,.
\end{align*}
However, if $n<0$, the integrand in the left side has not change of sign because $z$ is always nonzero, obtaining a contradiction. Similarly, if $n>0$ is not an even natural number, the left hand side is always positive because $n-1$ is not odd, obtaining again a contradiction.
\end{proof}

When $n\in\n$ is even, we prove the existence of closed $n$-elastic curves for arbitrary values of $\mu$, some of which are simple. See Figure \ref{cerradas}.

\begin{theorem}\label{t-closed} If $n\in \n$ is an even natural number, then there are closed $n$-elastic curves. Moreover, if $\mu\geq 0$ or
$$\mu<-\left(\frac{n+1}{n}\right)^\frac{n}{n+1}\,,$$
the curve is simple.
\end{theorem}
\begin{proof} The existence is obtained by choosing suitable conditions in \eqref{eq2}. Let $z_0=0$ and $\theta_0=\pi/2$. Recall that the maximal domain of the solution is $\r$. Define the functions 
\[\bar{y}(s)=  y(-s),\quad\quad \bar{z}(s)=  -z(-s),\quad\quad \bar{\theta}(s)= \pi-\theta(-s).\]
We check that  these functions satisfy \eqref{eq1}. For $\bar{y}$ and $\bar{z}$ is immediate. For $\bar{\theta}$, we have  
\[\bar{\theta}'(s)=\theta'(-s)=z(-s)^n+\mu=(-\bar{z}(s))^n+\mu=\bar{z}(s)^n+\mu,\]
where we have used that $n$ is even. Since the initial conditions at $s=0$ coincide with that of $\{y(s),z(s),\theta(s)\}$, uniqueness yields
\begin{align*}
y(s)&= y(-s),\\
 z(s)&=  -z(-s),\\
  \theta(s)&= \pi-\theta(-s).
  \end{align*}
This implies that the trace of $\gamma$ has a horizontal symmetry about the $y$-axis.  We also know from Proposition \ref{pr1} that $\gamma$ has a vertical symmetry, so we conclude that $\gamma$ is a closed curve. 

We now prove that if $\mu$ is nonnegative, then $\gamma$ is simple. Indeed, if $\mu\geq 0$, then $\theta$ is an increasing function. Then the monotonicity of $\theta$ together with the symmetries of $\gamma$ prove the claim.

In case that $\mu$ is negative, then the curves may be simple or not depending on the relation between $n$ and $\mu$. Due to $z_0=0$, we know that $\theta'(0)=\mu<0$, so $\theta$ is decreasing at $s=0$, i.e. it decreases from $\theta(0)=\theta_0=\pi/2$. Let $s_0>0$ be the first value where $\gamma$ is horizontal, so $\theta(s_0)=0$. Suppose that $s_1>0$ and $s_1\leq s_0$ is the first value, if it exists, such that $\theta'(s_1)=0$ and let $z_1=z(s_1)$ and $\theta_1=\theta(s_1)$.  From \eqref{ode1} at $s=0$, we deduce that $c=0$. On the other hand, 
$\theta'(s_1)=0=z_1^n+\mu$, hence $\mu=-z_1^n$. Then \eqref{ode1} gives  
$$\frac{n}{n+1}(-\mu)^{\frac{n+1}{n}}=\cos\theta_1\in (0,1].$$
For fixed $n$, above equality is not possible if
$$\frac{n}{n+1}(-\mu)^{\frac{n+1}{n}}>1\,.$$ 
This proves that, in those cases, $\theta'$ cannot vanish so $\theta$ is a decreasing function and $\theta(s_0)=0$. Thus, $\theta$ is monotonic and the curve $\gamma$ is simple.
\end{proof}

The initial conditions chosen in the proof of Theorem \ref{t-closed} imply that the Lagrange multiplier $\lambda$ is zero, that is, there is no constraint on the length of the curves. Thus critical curves for this case are usually referred to as \emph{free} $n$-elastic curves. The assumption $\lambda=0$ is quite reasonable since (free) $n$-elastic curves for $n\in\n$ cut the $y$-line perpendicularly. Clearly, this is a necessary condition for the associated curve to close. Moreover, we also need that after that point the curve goes backwards in the $y$-direction, instead of forward. From the parameterization \eqref{pc11}, we see that this happens, precisely, when $n$ is even. The existence of closed $n$-elastic curves other than pseudo-lemniscates is a huge difference, that deserves to be pointed out, between the cases $n$ odd (including the family of classical elastic curves) and $n$ even. Special mention deserves the existence of simple closed ones.

\begin{remark} For $n\in\n$ even, the variational characterization of Theorem \ref{crit} is local. Indeed, the energy $\mathbf{\Theta}$ acts on the space of curves satisfying $\kappa>\mu$. Therefore, in particular, the cuts with the $y$-line of the closed curves of Theorem \ref{t-closed} cannot be included. This prevents a contradiction with a rescaling argument to prove the nonexistence of closed critical curves for $\lambda=0$.
\end{remark}

\section{Classification of $n$-elastic curves}\label{5}

In this section we give a classification of all shapes of $n$-elastic curves with nonconstant curvature according to the parameters $n$ and $\mu$. 
Once we know $\theta$ in \eqref{eq1}, the function $y$ is completely determined, and consequently we may simplify \eqref{eq1} by the autonomous system  
$$\left\{\begin{split}
	\theta'&=z^n+\mu\\
	z'&=\sin\theta
\end{split}\right.\,.$$
We will investigate the phase portrait associated to this system.   Define  the vector field 
$$V(\theta,z)=(V_1(\theta,z),V_2(\theta,z))=(z^n+\mu,\sin\theta)$$
 whose orbit space is equivalent to that of $F$ in \eqref{F}. The orbits $(\theta(s),z(s))$ make a foliation by regular proper $C^1$ curves of the phase space except at the equilibrium points, which represent horizontal straight lines.
 
Let $\Omega$ be the domain of $V$. When $n\not\in\n$, $\Omega=\r\times (0,\infty)$ and if $n\in\n$, then $\Omega=\r^2$. In the first case, $\L=\r\times\{0\}$ will denote the boundary of $\Omega$. Let us observe that in the  phase portrait, the level curves are periodic in the $\theta$-direction with period $2\pi$.  
 
A special case will occur when the level curve is defined for any $s$ and, in addition, it is a graph on $\L$. Then the rank of $\theta$ is $\r$ and Proposition \ref{pr2} implies that $\gamma$ is either closed or invariant by a discrete group of horizontal translations. In this case, we obtain either orbitlike $n$-elastic curves (Figure \ref{figtipos}, (a) and Figure \ref{figtipos2}, (a)-(b)) or wavelike $n$-elastic curves (see Figure \ref{figtipos}, (c)-(h)). Apart from the closed curves obtained in Theorem \ref{t-closed} (Figure \ref{cerradas}), we also have pseudo-lemniscates (Figure \ref{figtipos}, (d)).

On the other hand, if the rank of $\theta$ is not $\r$, we may find complete curves with nonperiodic curvature $\kappa$. For instance, if the associated orbit $(\theta(s),z(s))$ goes to an equilibrium point we have borderline $n$-elastic curves (Figure \ref{figtipos}, (b)). By the uniqueness of the initial value problem, in this case, the parameter $s$ goes to $\pm\infty$. If the associated orbit does not approach an equilibrium point, these curves are catenary-like $n$-elastic curves (Figure \ref{figtipos2}, (c)-(h)). A complete deformation of $n$-elastic curves obtained varying $\lambda$, which goes from orbitlike to catenary-like curves, has been shown in Figure \ref{figtipos2}.

In other cases, the level curve meets $\L$ at two points (by the symmetry of Proposition \ref{pr1}) and the maximal domain of $\gamma$ is a bounded interval. See Figure \ref{cuts}.

The classification of $n$-elastic curves will consist in analyzing the geometry of the level curves of the vector field $V$. If we find the equilibrium points of $V$, $V(\theta,z)=(0,0)$, we deduce that $\theta=m\pi$, $m\in\z$. Thus the existence of equilibrium points depends on the equation $z^n+\mu=0$.  Whenever they exist, equilibrium points will be $\{(m\pi,\pm(-\mu)^{1/n}):m\in\z\}$. If $n\notin\n$ the $z$-coordinate is always positive so only the positive sign in the equilibrium points should be considered. Moreover, if $n\in\n$, and from Proposition \ref{pr6}, the equilibrium points $(m\pi,-(-\mu)^{1/n})$ can be viewed as the reflection about the $y$-line of those with positive $z$-coordinate. Therefore, we will restrict ourself to the positive sign. 

Let us compute now the partial derivatives of $V$, 
\begin{equation}\label{matriz}
\left(\begin{array}{cc}
\frac{\partial V_1}{\partial\theta}  &  \frac{\partial V_1}{\partial z} \\
& \\
\frac{\partial V_2}{\partial\theta} & \frac{\partial V_2}{\partial z} 
\end{array}\right)=\left(\begin{array}{cc}
0 &  nz^{n-1} \\
\cos\theta & 0
\end{array}\right).
\end{equation}
Depending on the sign of $\mu$ we have the following types of equilibrium points:
\begin{enumerate}
\item Case $\mu>0$. Then $n\in\z$ is odd and equilibria are unstable saddle points if $n>0$ and $m$ is even (or if $n<0$ and $m$ is odd) and centers if $n>0$ and $m$ is odd (or if $n<0$ and $m$ is even).
\item Case $\mu=0$. Then $n\in\n$. If $n=1$, we have unstable saddle points if $m$ is even and centers if $m$ is odd. If $n>1$, the matrix \eqref{matriz} is not diagonalizable with both the trace and the determinant vanishing.
\item Case $\mu<0$. Equilibria are unstable saddle points if $n>0$ and $m$ is even (or if $n<0$ and $m$ is odd) and centers if $n>0$ and $m$ is odd (or if $n<0$ and $m$ is even).
\end{enumerate}

We next distinguish between the cases $n\in\z$ even, $n\in\z$ odd and $n\notin\z$:

\begin{theorem}[Case $n\in\z$ even] Let $\gamma$ be a $n$-elastic curve with $n\in\z$ even. Then, either $\gamma$ is one of the closed curves of Theorem \ref{t-closed} (necessarily $n>0$) or, depending on the values of $\mu$ we have:
\begin{enumerate}
\item Case $\mu>0$. Then $\gamma$ is an orbitlike $n$-elastic curve.
\item Case $\mu=0$. Then $\gamma$ is either an orbitlike or a catenary-like $n$-elastic curve (the latter is only possible if $n<0$).
\item Case $\mu<0$. Then $\gamma$ is either an orbitlike, a borderline or a wavelike $n$-elastic curve.
\end{enumerate}
\end{theorem}
\begin{proof}
In all the cases, level curves are defined on $\r$ so all $n$-elastic curves are complete. As shown in Theorem \ref{t-closed}, when the initial conditions are $\theta_0=\pi/2$ and $z_0=0$, the $n$-elastic curve is closed for any value of $\mu$. Some of these curves are also simple. 

Apart from these cases, phase portraits for the case $\mu>0$, both with $n>0$ and $n<0$, present level curves that are entire graphs on $\L$ (Fig. \ref{figlevelneven1} and \ref{figlevelneven2}, (a)), obtaining orbitlike $n$-elastic curves. 
 
If $\mu=0$ and $n>0$, the level curves are again entire graphs on $\L$ (Fig. \ref{figlevelneven1}, (b)), obtaining orbitlike $n$-elastic curves. If $\mu=0$ and $n<0$, the level curves are graphs on $\L$ when the value $z_0$ is close to $0$ (Fig. \ref{figlevelneven2}, (b)) and, hence, we have curves of orbitlike type. However, if $z_0$ increases and $\theta_0=0$, level curves are graphs on small bounded intervals of $\L$. This means that $\gamma$ is of catenary-type, which may be simple or not. In case that the curve is simple, the curve must be convex.  

Finally, assume $\mu<0$ (Fig. \ref{figlevelneven1} and \ref{figlevelneven2}, (c)). Orbitlike $n$-elastic curves appear when $z_0$ is sufficiently big ($n>0$) or close to zero ($n<0$). Around the critical points of $V$ which are centers, and whenever $z_0$ is closed to the value of the critical point, level curves correspond with wavelike $n$-elastic curves because the rank of $\theta$ lies on a bounded interval. Moreover, there are also borderline $n$-elastic curves asymptotic to the horizontal straight line of the equilibrium point of saddle type.  
\end{proof}

\begin{figure}[hbtp]
\centering
\makebox[\textwidth][c]{
\begin{subfigure}[b]{0.3\linewidth}
\includegraphics[width=\textwidth]{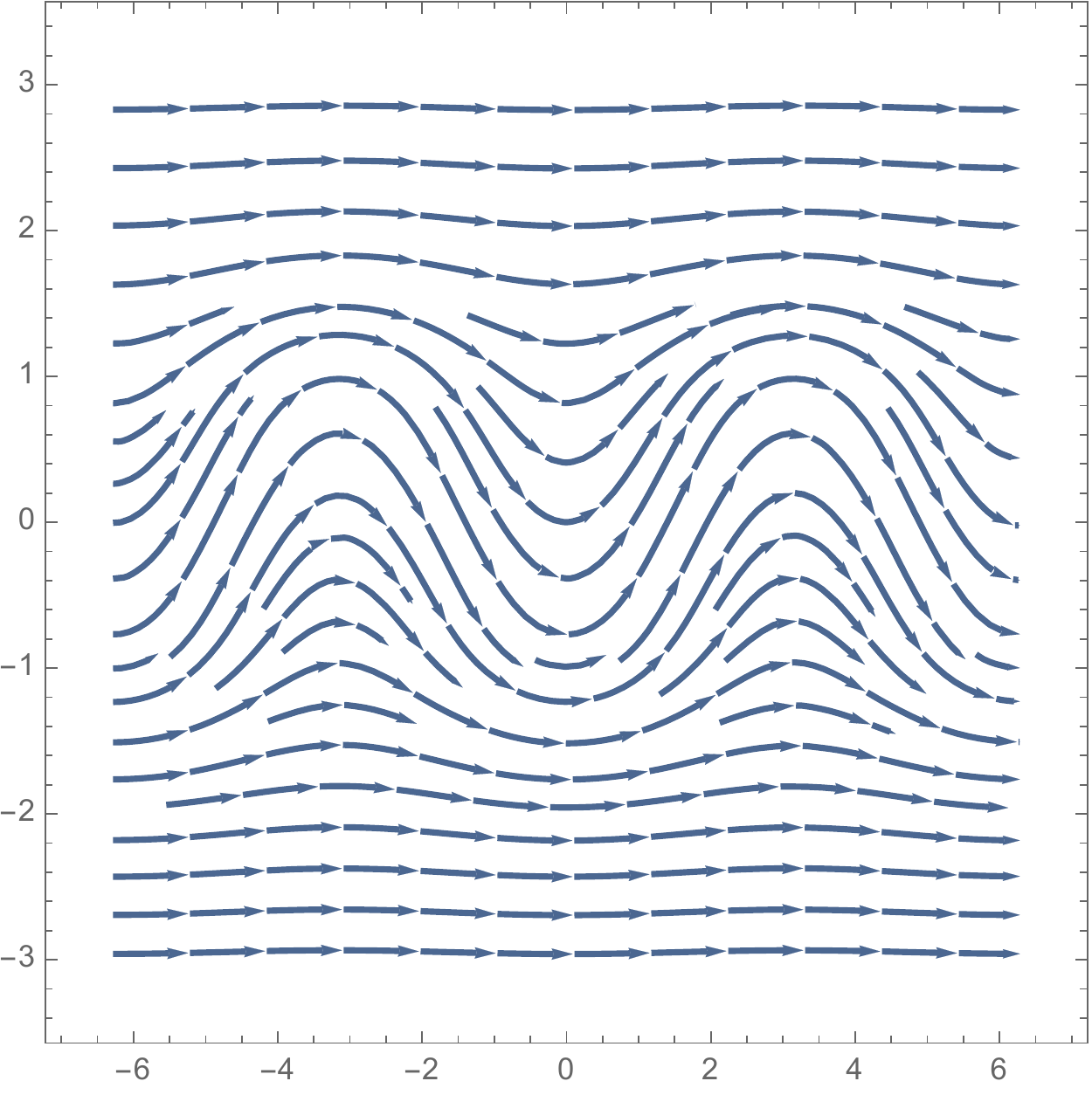}
\caption{$\mu=1$}
\end{subfigure}
\,\quad
\begin{subfigure}[b]{0.3\linewidth}
\includegraphics[width=\textwidth]{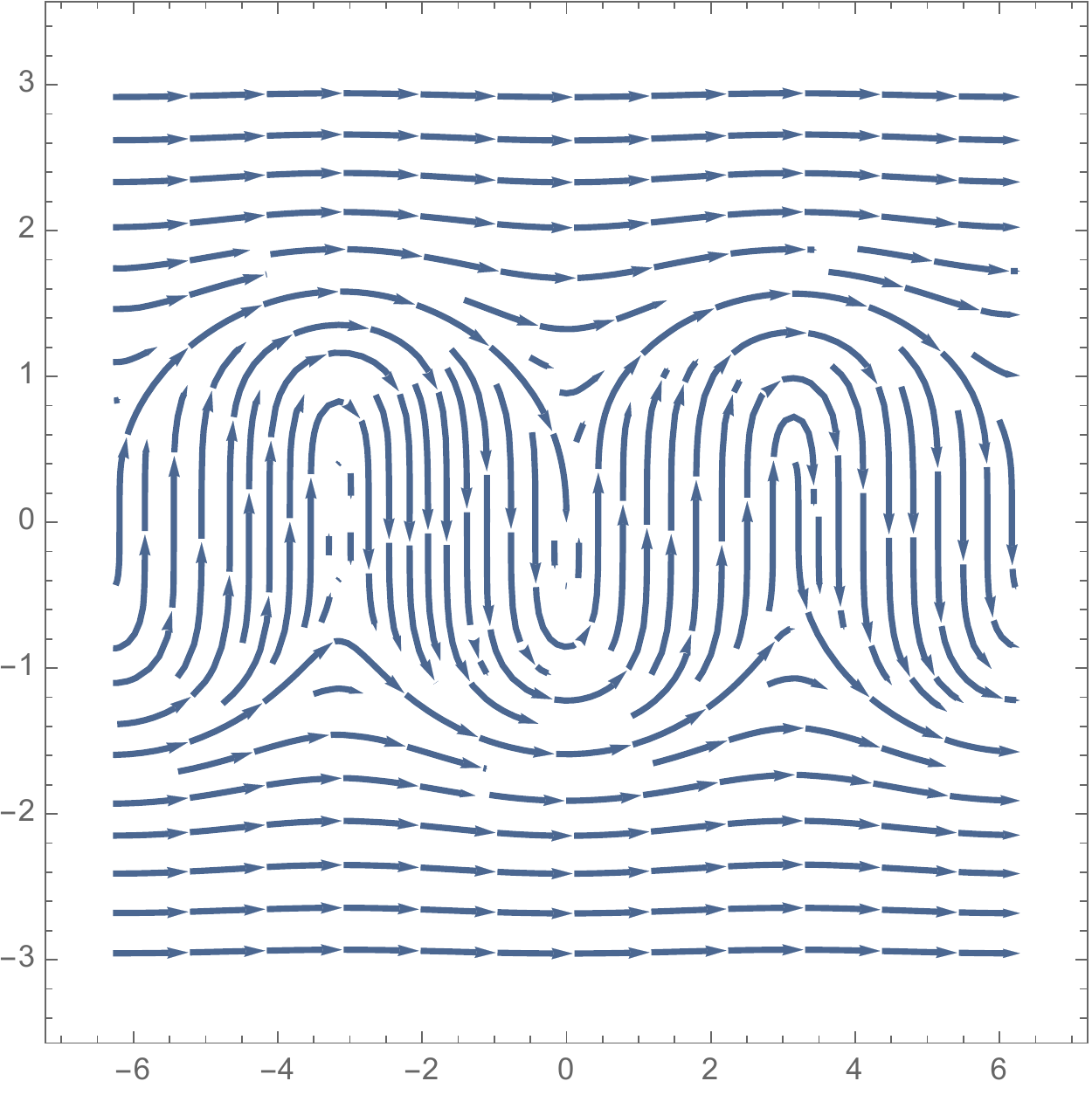}
\caption{$\mu=0$}
\end{subfigure}
\,\quad
\begin{subfigure}[b]{0.3\linewidth}
\includegraphics[width=\textwidth]{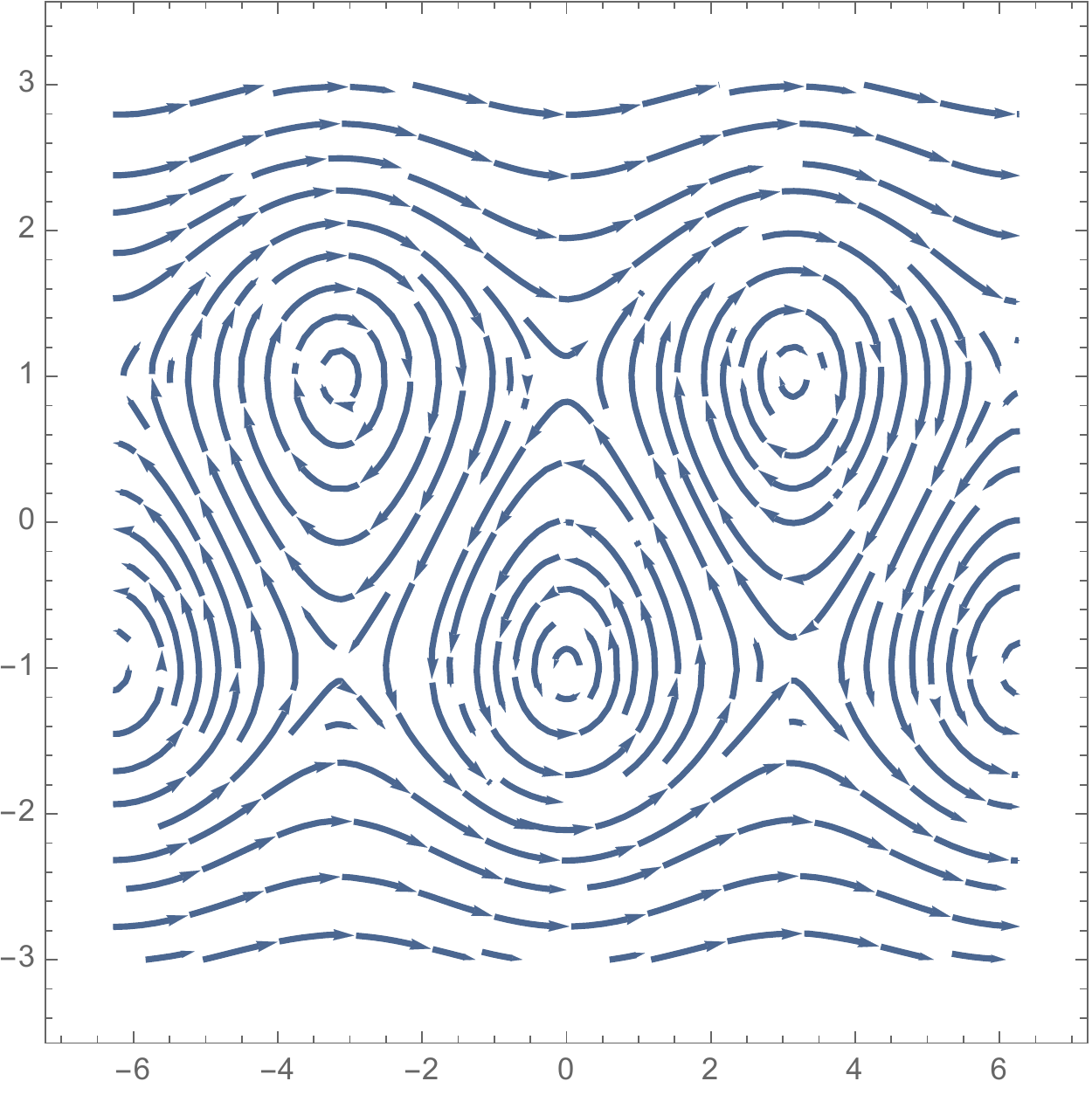}
\caption{$\mu=-1$}
\end{subfigure}}
\caption{Level curves of $V(u,v)$ for the case $n>0$ even. Here $n=4$.}\label{figlevelneven1}
\end{figure}

\begin{figure}[hbtp]
\centering
\makebox[\textwidth][c]{
\begin{subfigure}[b]{0.3\linewidth}
\includegraphics[width=\textwidth]{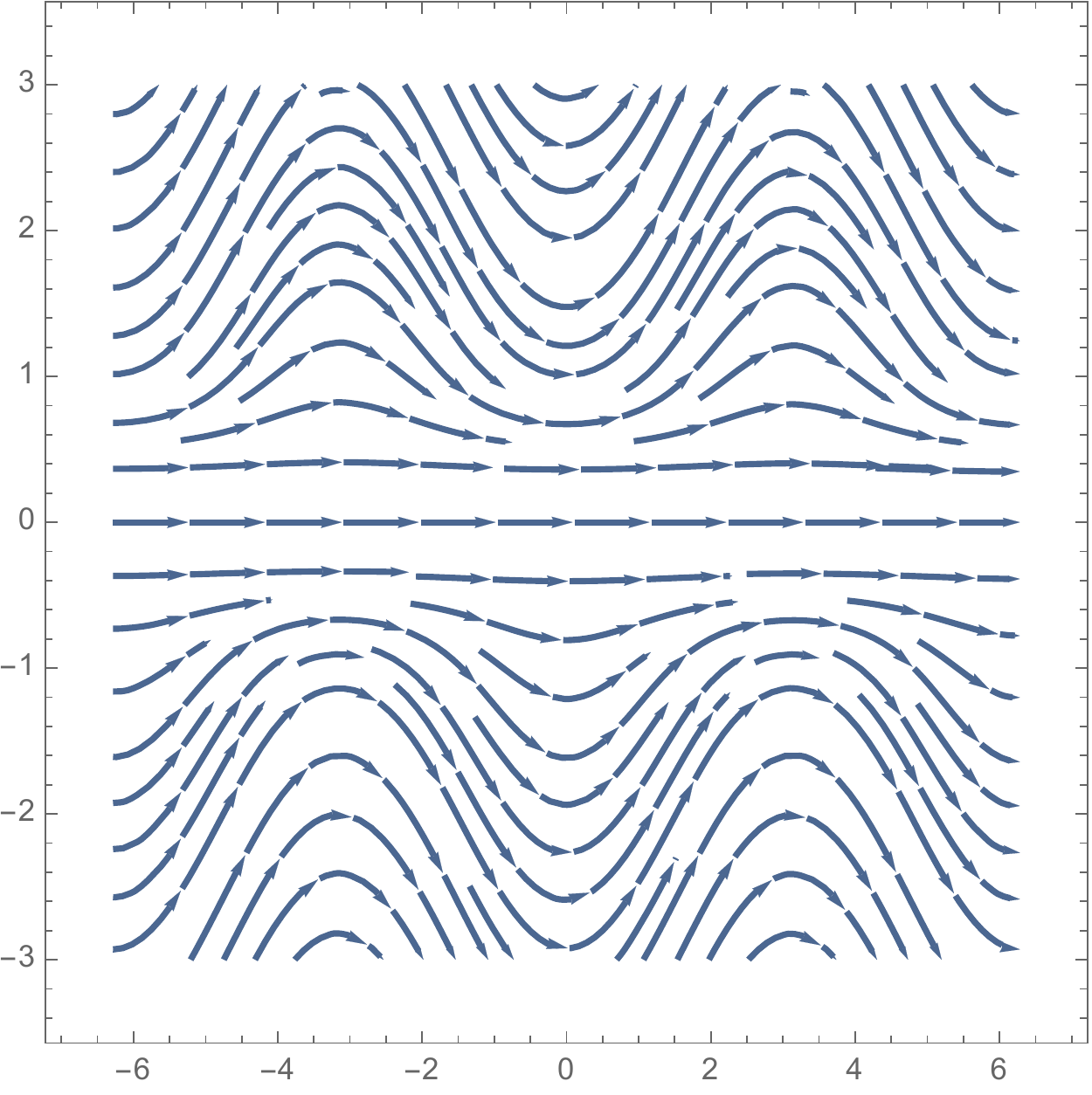}
\caption{$\mu=1$}
\end{subfigure}
\,\quad
\begin{subfigure}[b]{0.3\linewidth}
\includegraphics[width=\textwidth]{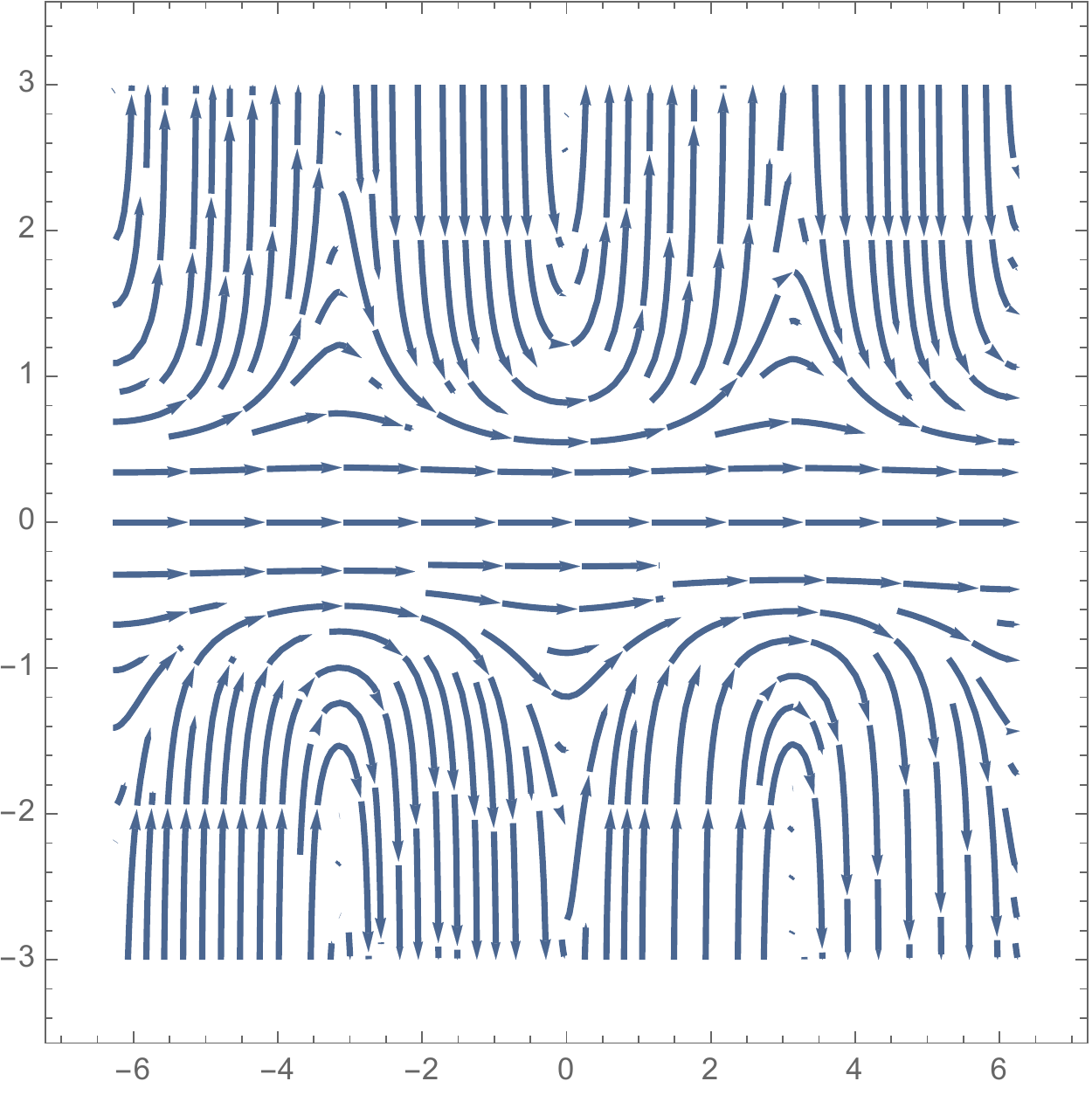}
\caption{$\mu=0$}
\end{subfigure}
\,\quad
\begin{subfigure}[b]{0.3\linewidth}
\includegraphics[width=\textwidth]{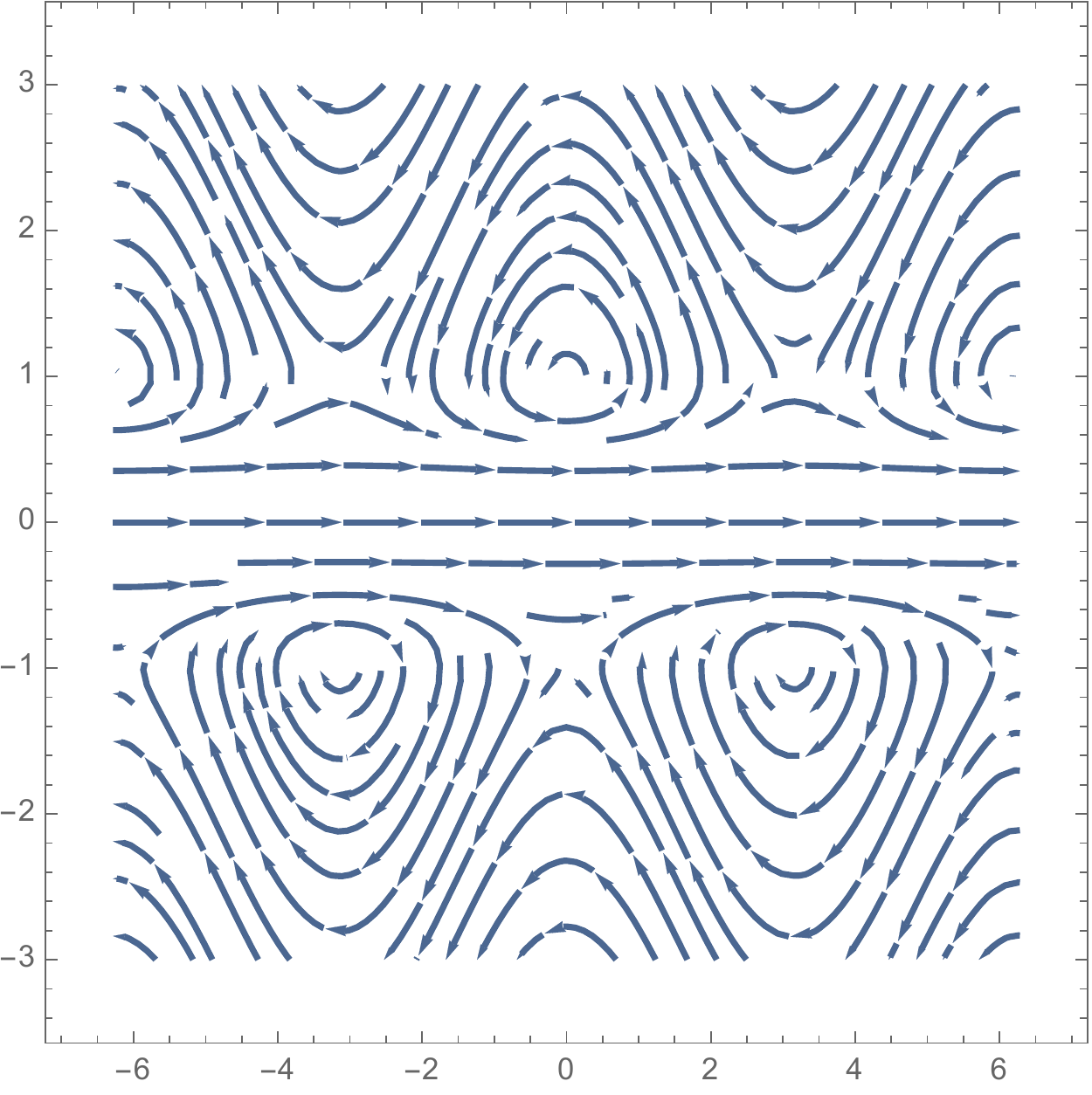}
\caption{$\mu=-1$}
\end{subfigure}}
\caption{Level curves of $V(u,v)$ for the case $n<0$ even. Here $n=-4$.}\label{figlevelneven2}
\end{figure}

\begin{theorem}[Case $n\in\z$ odd] Let $\gamma$ be a $n$-elastic curve with $n\in\z$ odd. Then:
\begin{enumerate}
\item If $n<0$ and $\mu=0$, $\gamma$ is either an orbitlike or a catenary-like $n$-elastic curve.
\item In the rest of the cases, $\gamma$ is either an orbitlike, a borderline or a wavelike $n$-elastic curve. 
\end{enumerate}
\end{theorem}
\begin{proof} As in the even case, all level curves are defined on $\r$ so associated $n$-elastic curves are complete. Since $n$ is odd, for any value of $\mu$ there are critical points of $V$, with the exception of $n<0$ and $\mu=0$. For the rest of the cases (Fig. \ref{figlevelodd1}, (a)-(c), and Fig. \ref{figlevelodd2}, (a) and (c)) and for values $z_0$ sufficiently big (if $n>0$) or close to zero (if $n<0$), level curves of $V$ are entire graphs on $\L$, hence, the corresponding $n$-elastic curve is of orbitlike type. The level curves around the critical points which are centers represent $n$-elastic curves on which the angle $\theta$ varies in some bounded interval, thus, they are wavelike $n$-elastic curves. Moreover, there are orbits approaching saddle critical points and, hence, representing borderline $n$-elastic curves.
 
In the case $n<0$ and $\mu=0$ (Fig. \ref{figlevelodd2}, (b)), besides orbitlike $n$-elastic curves ($z_0$ sufficiently close to zero), we also have catenary-like $n$-elastic curves when $\theta_0=0$ for values $z_0$ far from zero. In this case, there are no critical points and so there are not borderline nor wavelike $n$-elastic curves.
\end{proof}

\begin{figure}[hbtp]
\centering
\makebox[\textwidth][c]{
\begin{subfigure}[b]{0.3\linewidth}
\includegraphics[width=\textwidth]{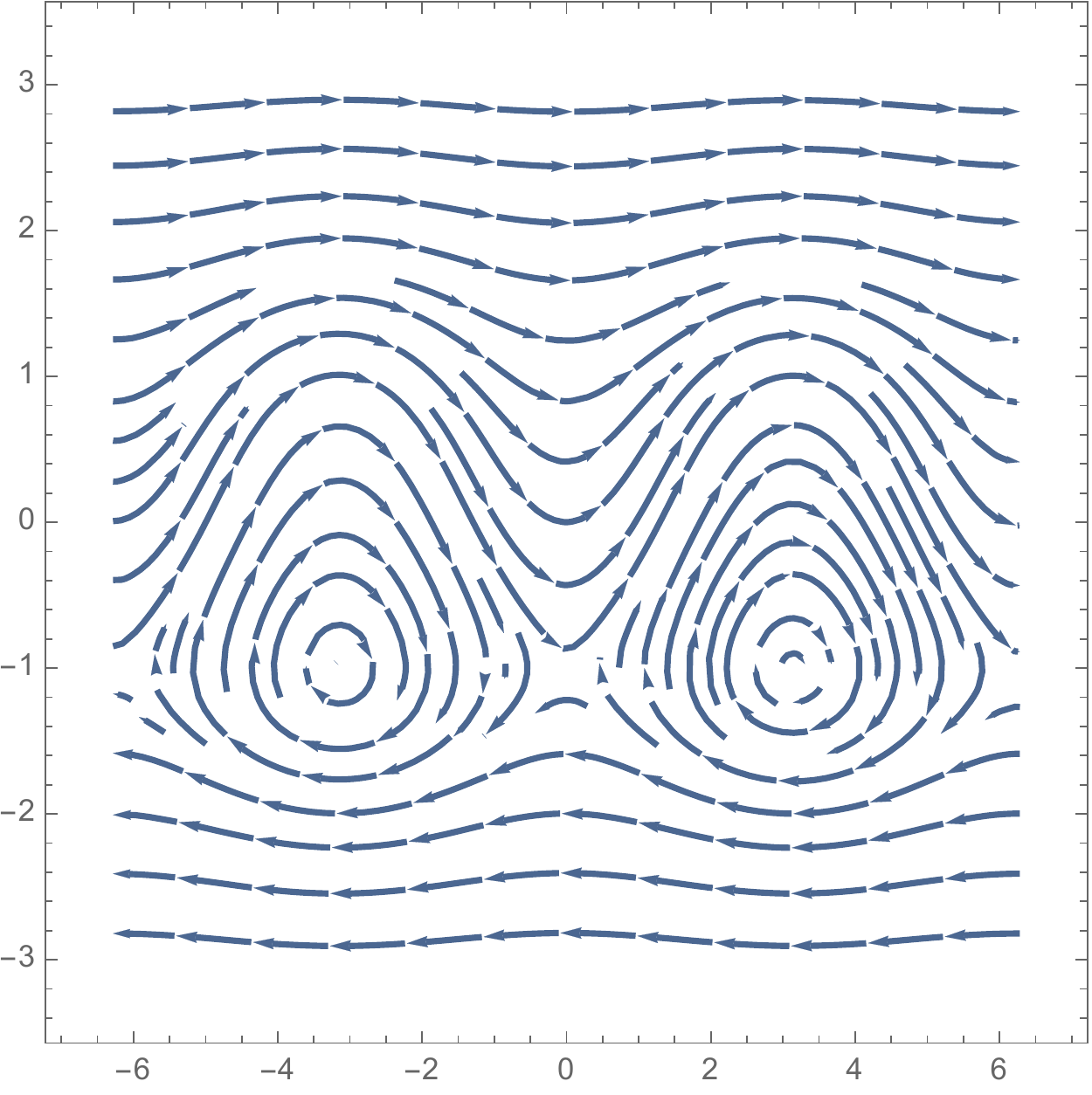}
\caption{$\mu=1$}
\end{subfigure}
\,\quad
\begin{subfigure}[b]{0.3\linewidth}
\includegraphics[width=\textwidth]{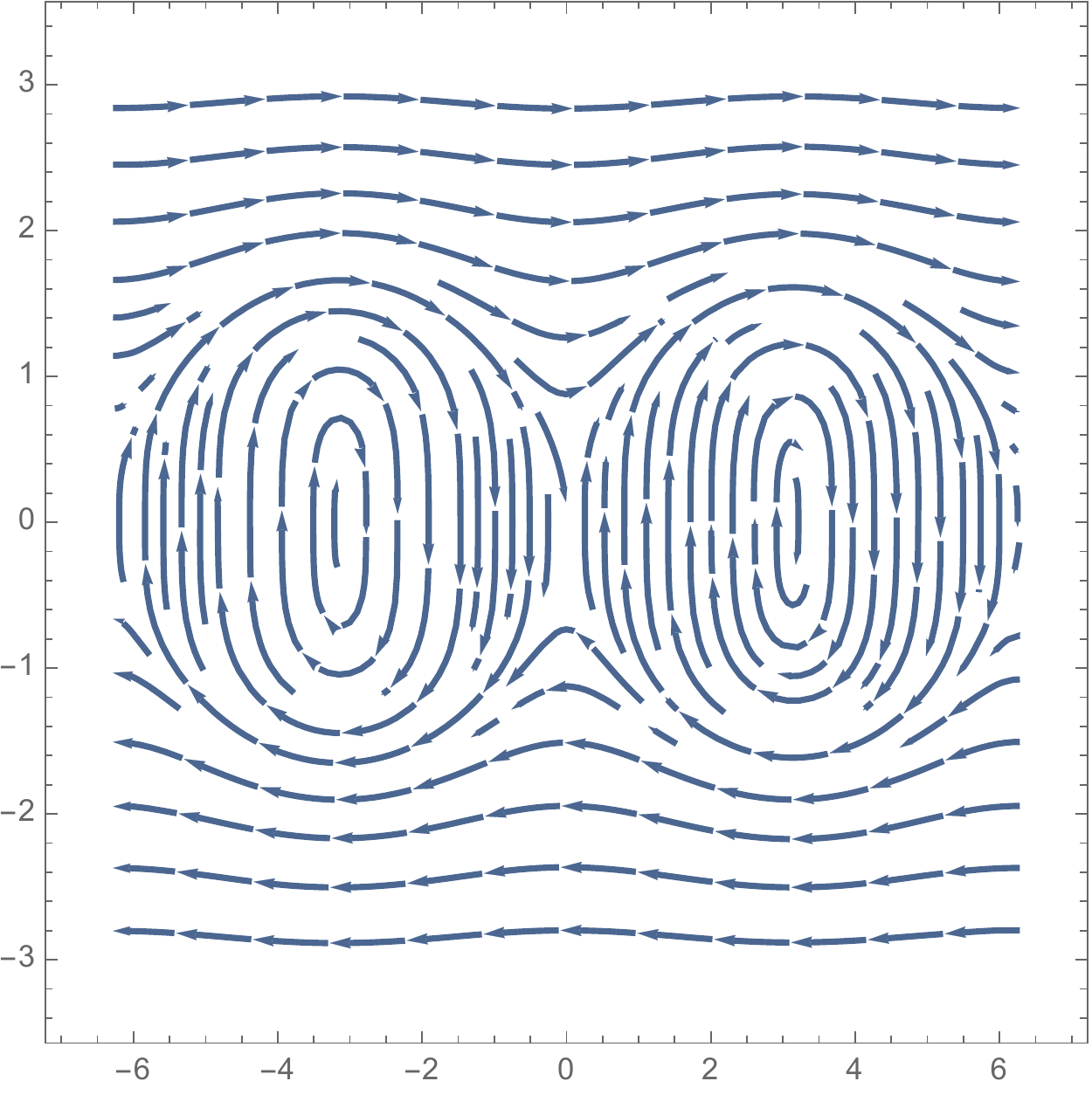}
\caption{ $\mu=0$}
\end{subfigure}
\,\quad
\begin{subfigure}[b]{0.3\linewidth}
\includegraphics[width=\textwidth]{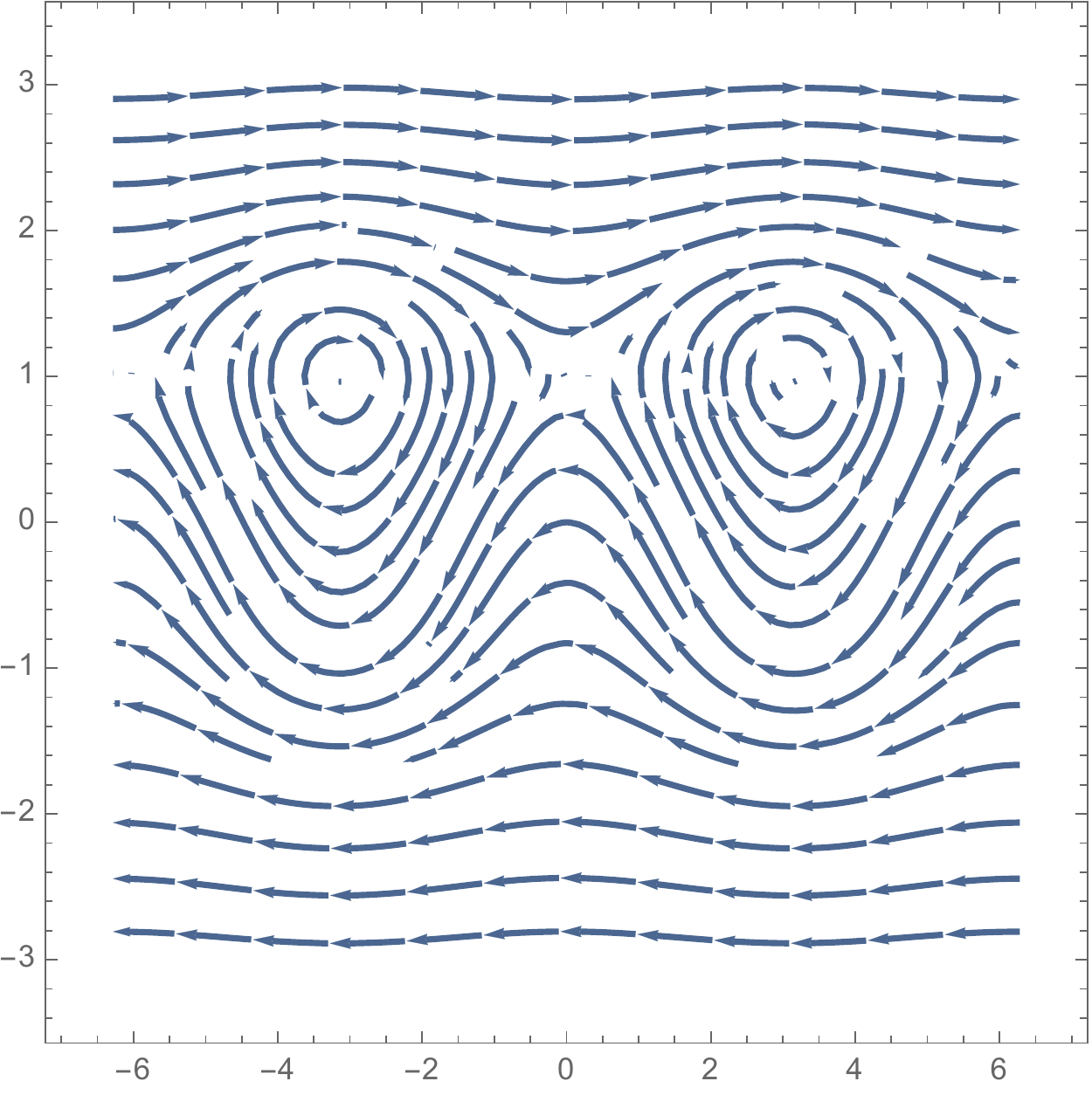}
\caption{$\mu=-1$}
\end{subfigure}}
\caption{Level curves of $V(u,v)$ for the case $n>0$ odd. Here $n=3$.}\label{figlevelodd1}
\end{figure}

\begin{figure}[hbtp]
\centering
\makebox[\textwidth][c]{
\begin{subfigure}[b]{0.3\linewidth}
\includegraphics[width=\textwidth]{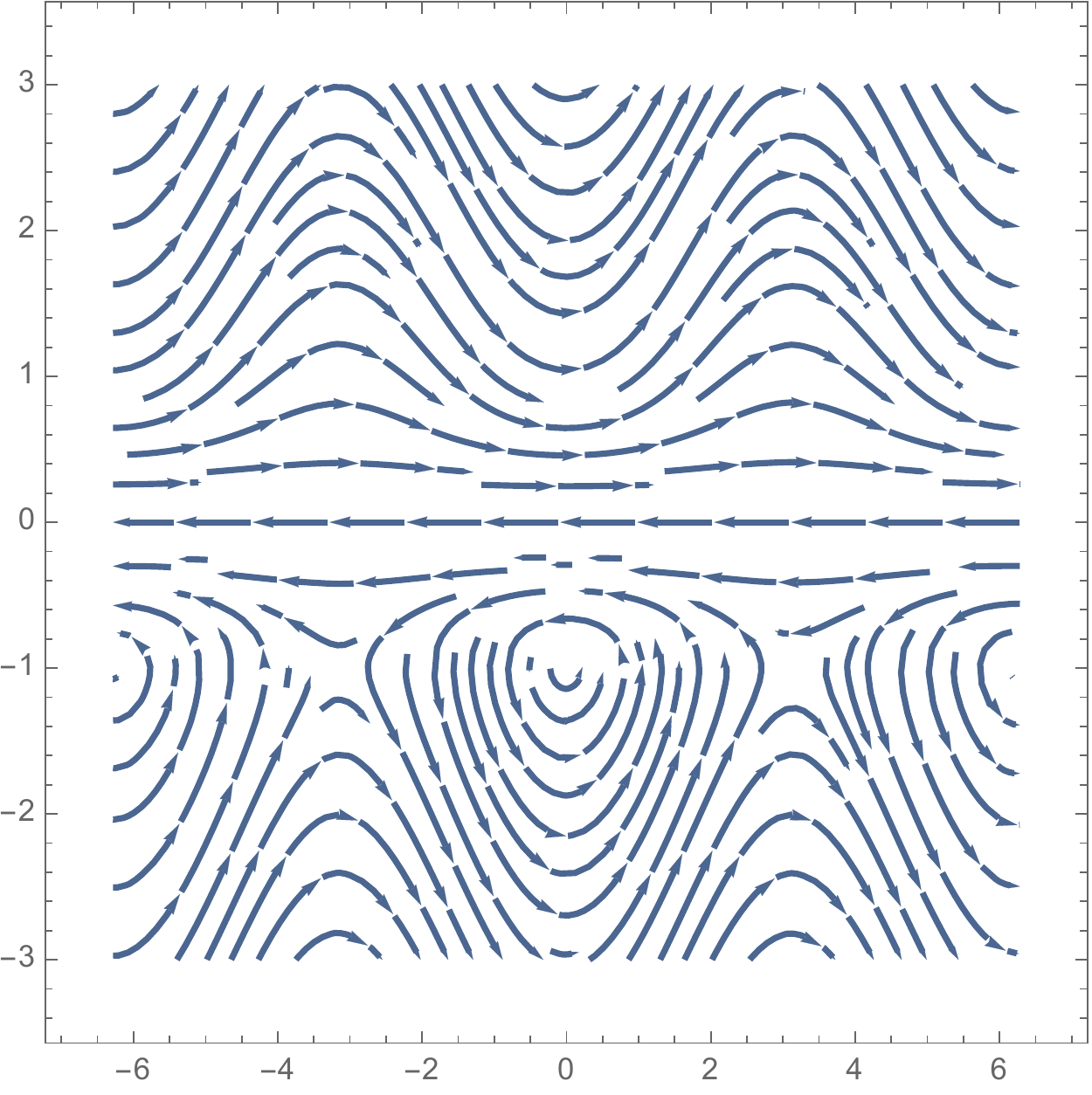}
\caption{$\mu=1$}
\end{subfigure}
\,\quad
\begin{subfigure}[b]{0.3\linewidth}
\includegraphics[width=\textwidth]{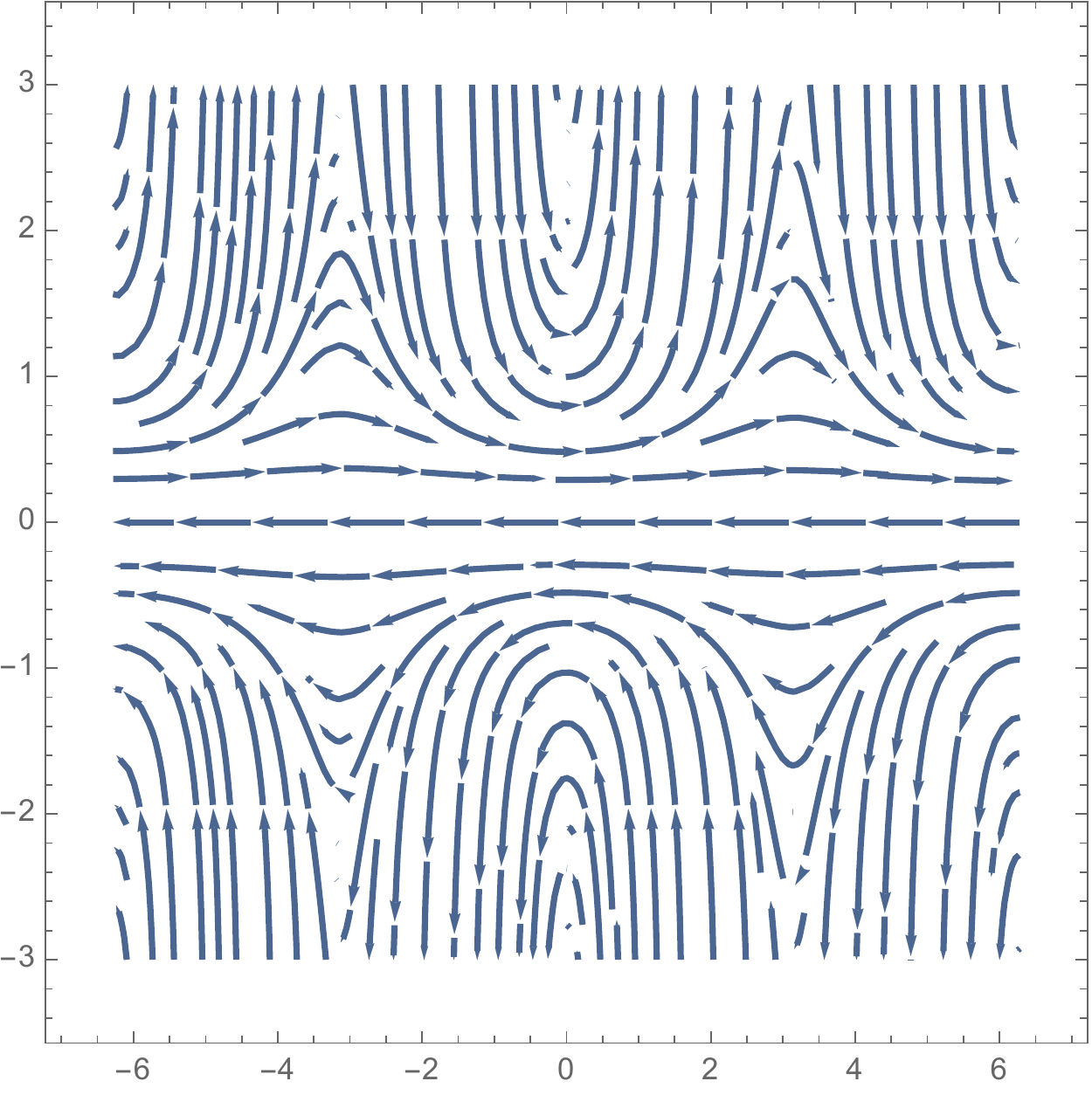}
\caption{$\mu=0$}
\end{subfigure}
\,\quad
\begin{subfigure}[b]{0.3\linewidth}
\includegraphics[width=\textwidth]{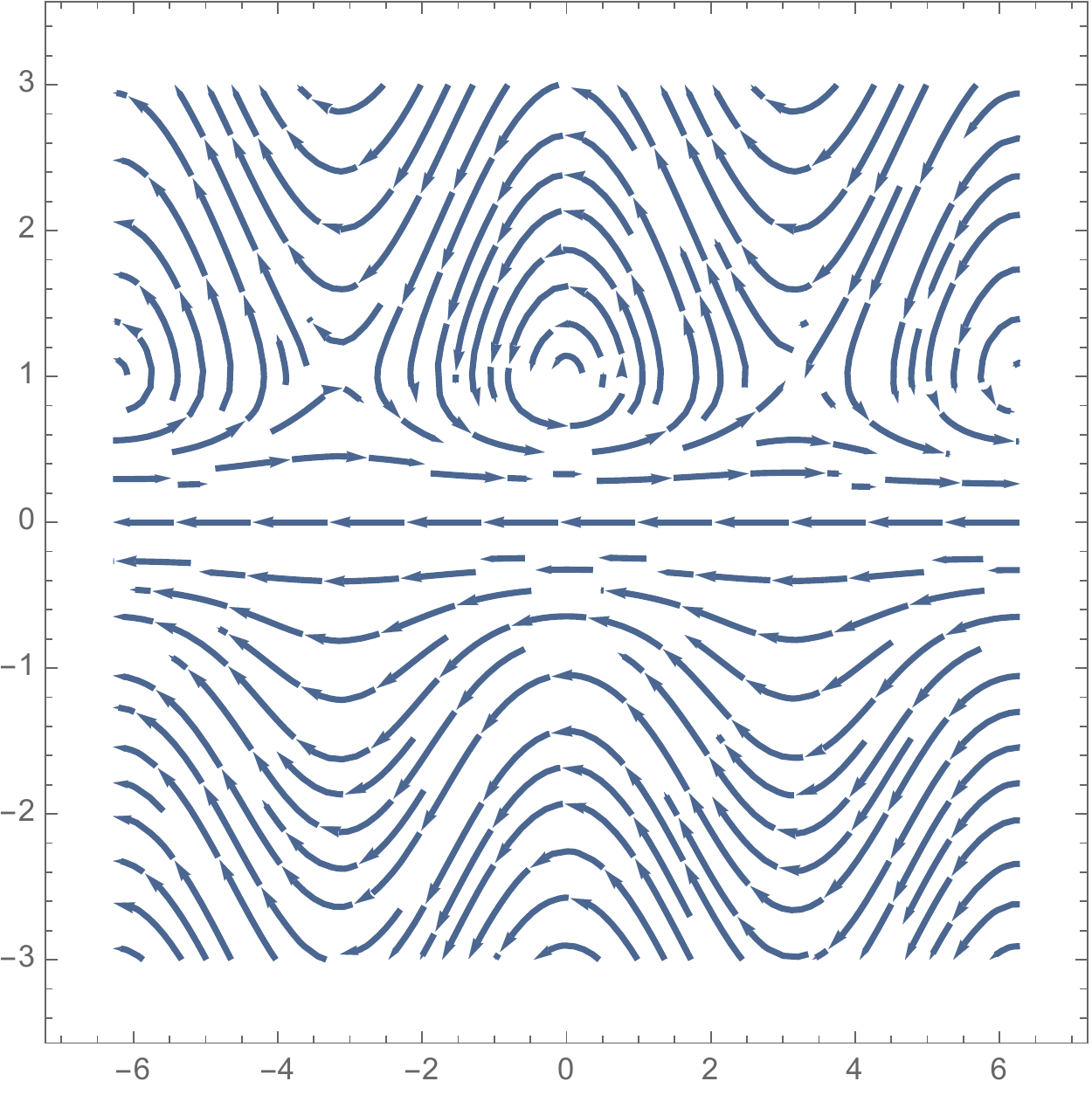}
\caption{ $\mu=-1$}
\end{subfigure}}
\caption{Level curves of $V(u,v)$ for the case $n<0$ odd. Here $n=-3$.}\label{figlevelodd2}
\end{figure}

\begin{theorem}[Case $n\notin\z$] Let $\gamma$ be a $n$-elastic curve with $n\notin\z$. If $\gamma$ is complete, depending on the values of $\mu$ we have:
\begin{enumerate}
\item Case $\mu>0$. Then $\gamma$ is an orbitlike $n$-elastic curve.
\item Case $\mu=0$. Then $\gamma$ is either an orbitlike or a catenary-like $n$-elastic curve (the latter is only possible if $n\leq -1$).
\item Case $\mu<0$. Then $\gamma$ is either an orbitlike, a borderline or a wavelike $n$-elastic curve.
\end{enumerate}
Moreover, for $n>-1$ and any value of $\mu$, we also have noncomplete $n$-elastic curves whose end points intersect the $y$-line.
\end{theorem}
\begin{proof}
We first note that when $n\leq -1$ or $n>-1$ and $z_0$ sufficiently big, level curves are defined on $\r$. However, for $n>-1$ and $z_0$ close to zero, level curves meet $\L$ at two points (see Fig. \ref{figleveln>0} at $\theta_0=\pi$) and, hence, $\gamma$ intersects the $y$-line at two points. This proves the second statement.

We focus now on complete $n$-elastic curves, i.e. on level curves defined on $\r$. If $\mu>0$ (Fig. \ref{figleveln>0} and \ref{figleveln<0}, (a)), there are no equilibria and level curves are entire graphs on $\L$ so we obtain orbitlike $n$-elastic curves. The case $\mu=0$ and $n>-1$ (Fig. \ref{figleveln>0}, (b)) is similar and we also obtain orbitlike type curves.

If $\mu=0$ and $n\leq -1$ (Fig. \ref{figleveln<0}, (b)) there are no equilibria. For $z_0$ small enough we have entire graphs on $\L$ producing orbitlike $n$-elastic curves, while if $z_0$ is big enough and $\theta_0=0$, we have catenary-like $n$-elastic curves.

Finally, if $\mu<0$, we may have both centers and saddle critical points. Therefore, a similar argument as in previous cases shows the existence of orbitlike, borderline and wavelike $n$-elastic curves.
\end{proof}

\begin{figure}[hbtp]
\centering
\begin{subfigure}[b]{0.3\linewidth}
\includegraphics[width=\textwidth]{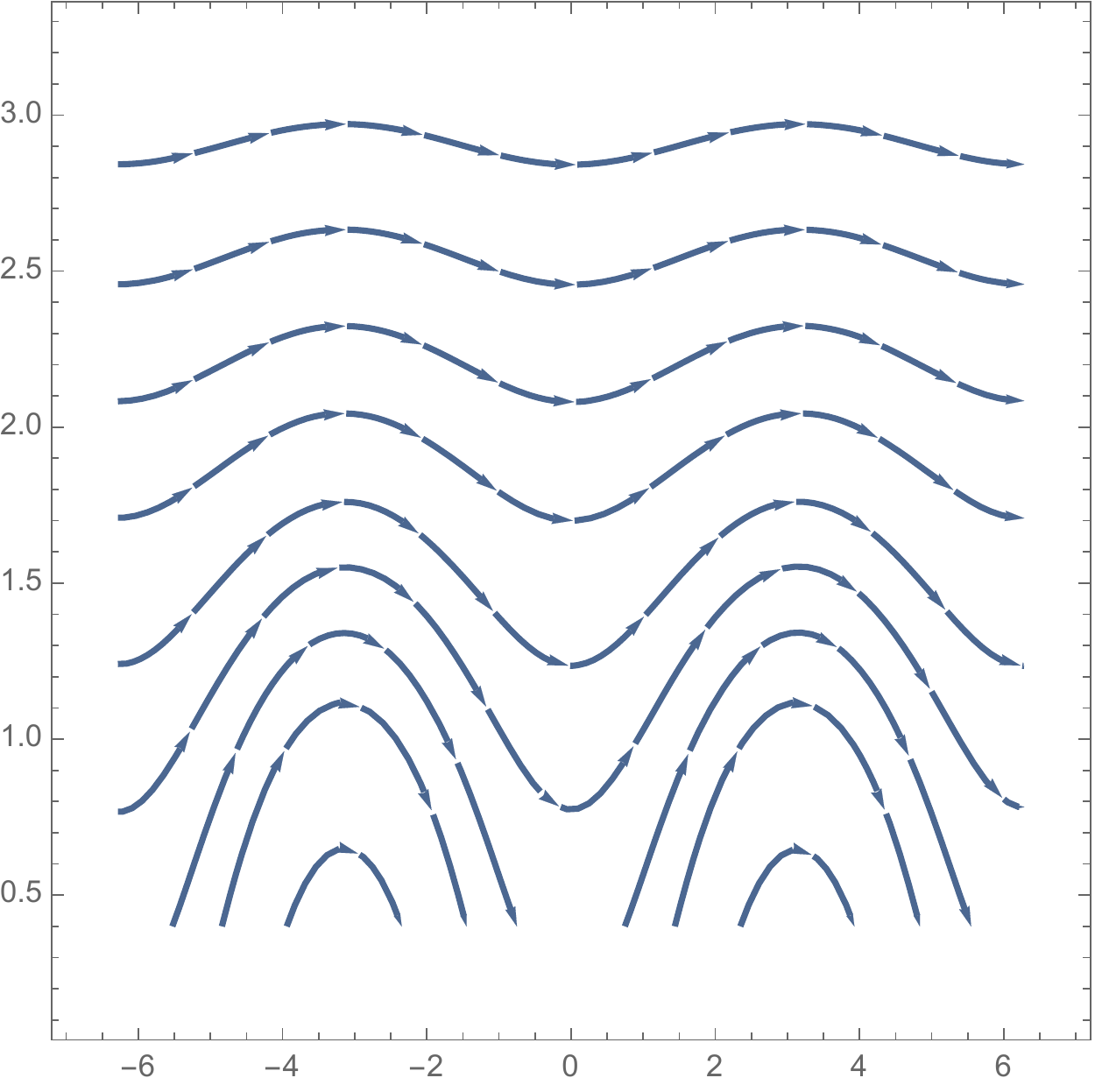}
\caption{$n=2.5$, $\mu=1$}
\end{subfigure}
\,\quad
\begin{subfigure}[b]{0.3\linewidth}
\includegraphics[width=\textwidth]{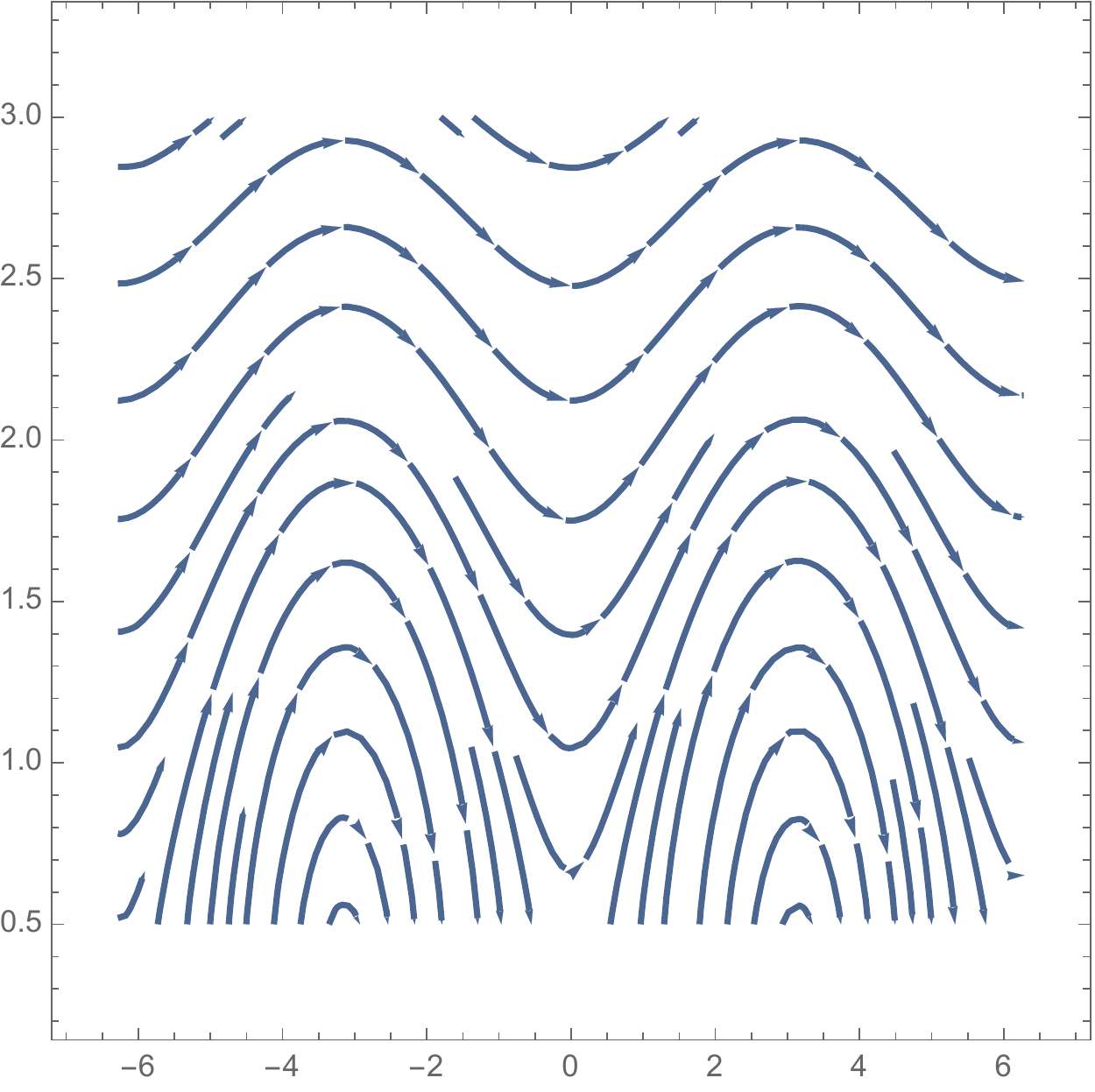}
\caption{$n=1.5$, $\mu=0$}
\end{subfigure}
\,\quad
\begin{subfigure}[b]{0.3\linewidth}
\includegraphics[width=\textwidth]{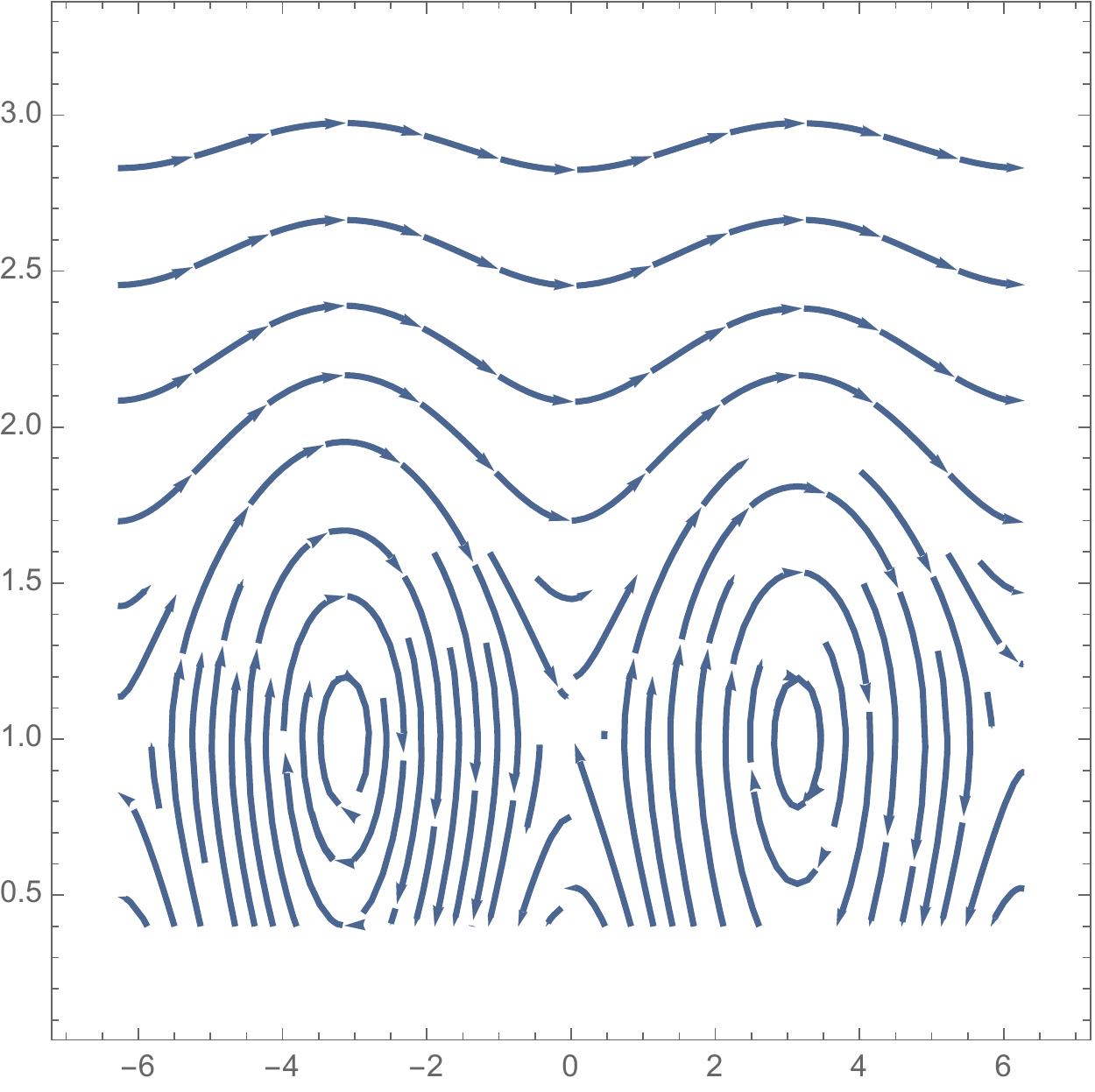}
\caption{$n=2.5$, $\mu=-1$}
\end{subfigure}
\caption{Level curves of $V(u,v)$ for the case $n>0$ and $n\notin\z$.}\label{figleveln>0}
\end{figure}

\begin{figure}[hbtp]
\centering
\begin{subfigure}[b]{0.3065\linewidth}
\includegraphics[width=\textwidth]{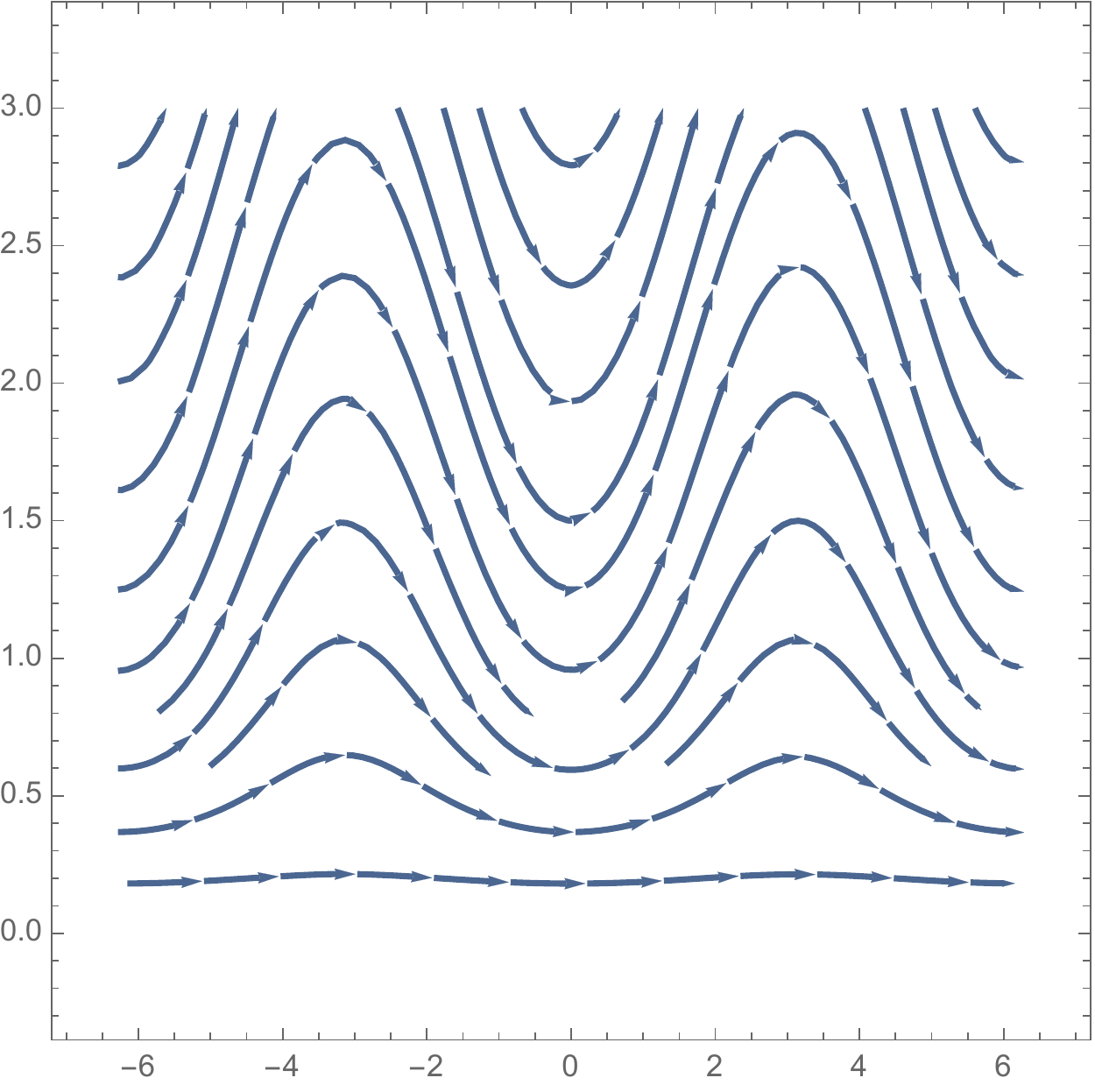}
\caption{$n=-2.5$, $\mu=1$}
\end{subfigure}
\,\quad
\begin{subfigure}[b]{0.3\linewidth}
\includegraphics[width=\textwidth]{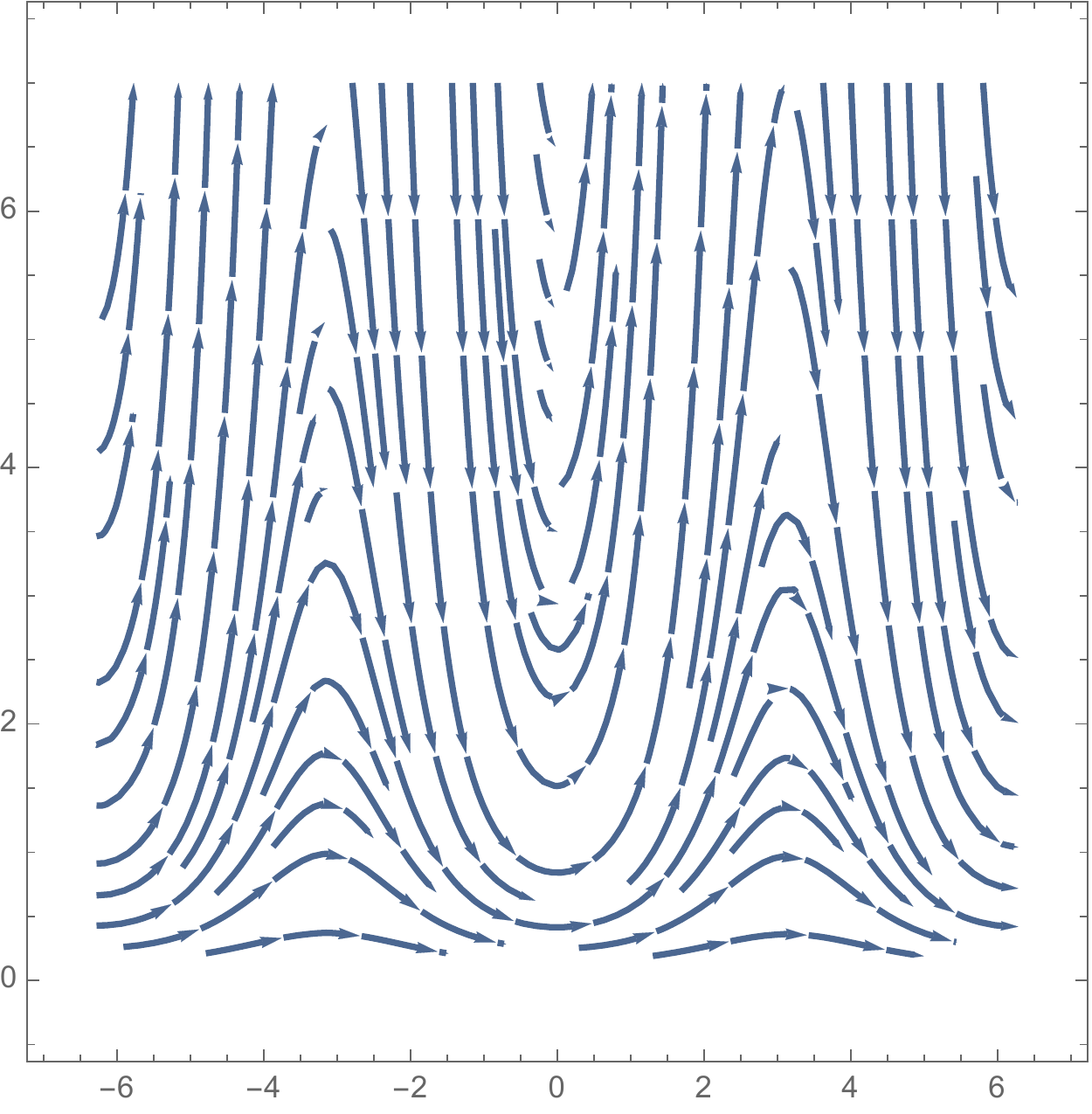}
\caption{$n=-1.5$, $\mu=0$}
\end{subfigure}
\,\quad
\begin{subfigure}[b]{0.3\linewidth}
\includegraphics[width=\textwidth]{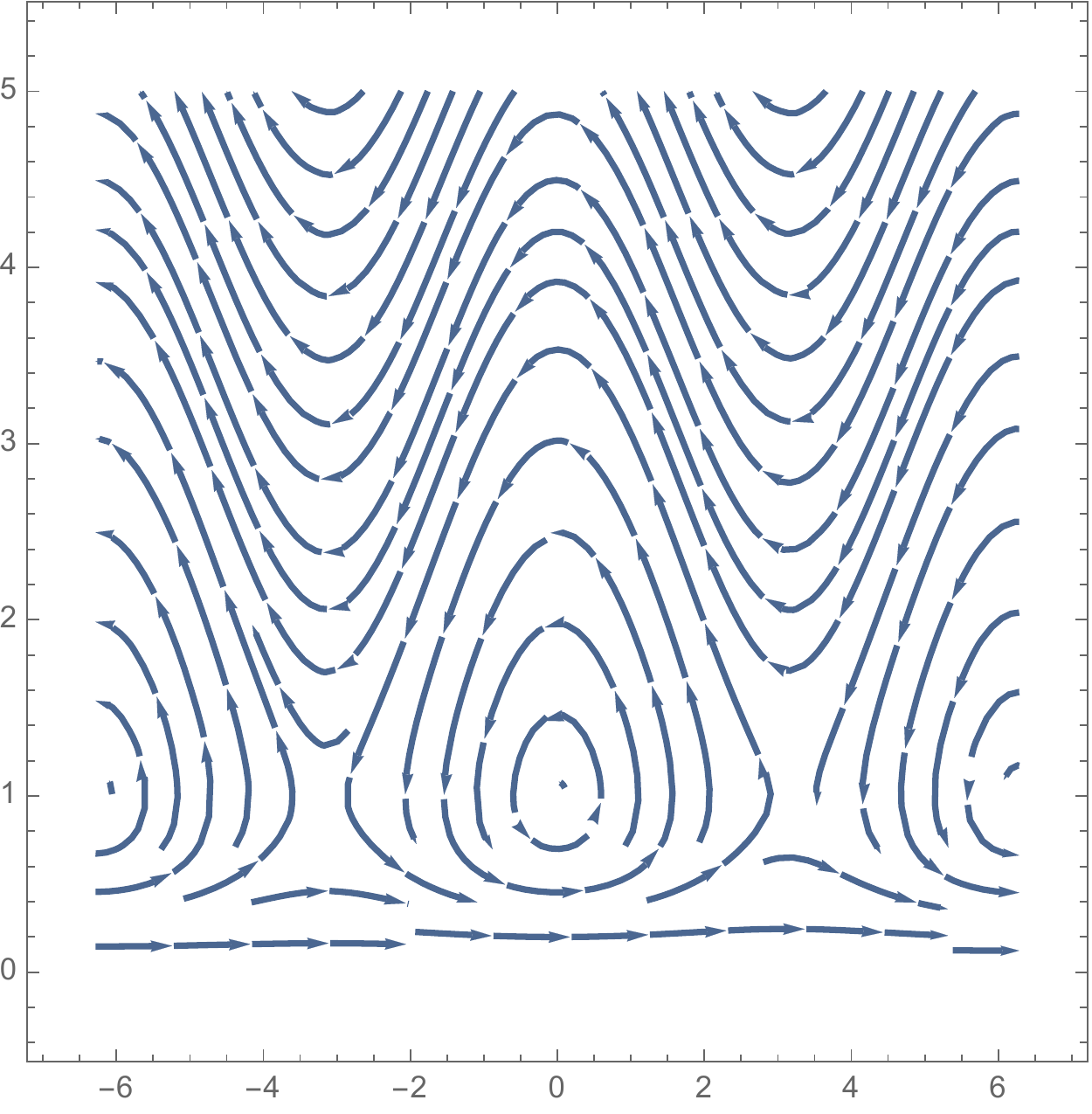}
\caption{$n=-2.5$, $\mu=-1$}
\end{subfigure}
\caption{Level curves of $V(u,v)$ for the case $n<0$ and $n\notin\z$.}\label{figleveln<0}
\end{figure}

In each of the cases discussed above, we obtained the phase portrait of $V$ using the \textit{Mathematica} software. The images were generated using the \texttt{StreamPlot} command.
   

\section*{Acknowledgements}
Rafael L\'opez has been partially supported by the grant no. MTM2017-89677-P, MINECO/ AEI/FEDER, UE.  The authors would like to thank the referee for carefully reviewing the paper.

\end{document}